\documentclass[oneside,reqno,english]{amsart}
\usepackage{lmodern}
\usepackage[T1]{fontenc}
\usepackage[latin9]{inputenc}
\usepackage{babel}
\usepackage{varioref}
\usepackage{float}
\usepackage{amstext}
\usepackage{amsthm}
\usepackage{amssymb}
\usepackage{graphicx}
\usepackage[unicode=true,pdfusetitle,
 bookmarks=true,bookmarksnumbered=false,bookmarksopen=false,
 breaklinks=false,pdfborder={0 0 1},backref=false,colorlinks=false]
 {hyperref}
\usepackage{breakurl}

\makeatletter

\floatstyle{ruled}
\newfloat{algorithm}{tbp}{loa}
\providecommand{\algorithmname}{Algorithm}
\floatname{algorithm}{\protect\algorithmname}

\numberwithin{equation}{section}
\numberwithin{figure}{section}
\theoremstyle{plain}
\newtheorem{thm}{\protect\theoremname}[section]
\theoremstyle{definition}
\newtheorem{example}[thm]{\protect\examplename}
\theoremstyle{definition}
\newtheorem{defn}[thm]{\protect\definitionname}
\theoremstyle{plain}
\newtheorem{lem}[thm]{\protect\lemmaname}
\theoremstyle{plain}
\newtheorem{lyxalgorithm}[thm]{\protect\algorithmname}
\theoremstyle{remark}
\newtheorem{rem}[thm]{\protect\remarkname}
\theoremstyle{plain}
\newtheorem{prop}[thm]{\protect\propositionname}
\theoremstyle{plain}
\newtheorem{assumption}[thm]{\protect\assumptionname}
\theoremstyle{plain}
\newtheorem{cor}[thm]{\protect\corollaryname}

\newcommand{\dom}{\mbox{\rm dom}}

\newcommand{\epi}{\mbox{\rm epi}\,}

\newcommand{\spn}{\mbox{\rm span}}
\newcommand{\cl}{\mbox{\rm cl}}

\makeatother

\providecommand{\algorithmname}{Algorithm}
\providecommand{\assumptionname}{Assumption}
\providecommand{\corollaryname}{Corollary}
\providecommand{\definitionname}{Definition}
\providecommand{\examplename}{Example}
\providecommand{\lemmaname}{Lemma}
\providecommand{\propositionname}{Proposition}
\providecommand{\remarkname}{Remark}
\providecommand{\theoremname}{Theorem}

\begin{document}
\title[Distributed Dykstra algorithm convergence rates]{Linear and sublinear convergence rates for a subdifferentiable distributed deterministic asynchronous Dykstra's algorithm} 

\begin{abstract}
In \cite{Pang_Dist_Dyk,Pang_sub_Dyk}, we designed a distributed deterministic
asynchronous algorithm for minimizing the sum of subdifferentiable
and proximable functions and a regularizing quadratic on time-varying
graphs based on Dykstra's algorithm, or block coordinate dual ascent.
Each node in the distributed optimization problem is the sum of a
known regularizing quadratic and a function to be minimized. In this
paper, we prove sublinear convergence rates for the general algorithm,
and a linear rate of convergence if the function on each node is smooth
with Lipschitz gradient. Our numerical experiments also verify these
rates.
\end{abstract}

\author{C.H. Jeffrey Pang}

\thanks{C.H.J. Pang acknowledges grant R-146-000-214-112 from the Faculty
of Science, National University of Singapore. }

\curraddr{Department of Mathematics\\ 
National University of Singapore\\ 
Block S17 08-11\\ 
10 Lower Kent Ridge Road\\ 
Singapore 119076 }

\email{matpchj@nus.edu.sg}

\date{\today{}}


\maketitle
\tableofcontents{}

\section{Introduction}

Let $V$ and $\bar{E}$ be finite sets. Define the set $\mathbf{X}:=X_{1}\times\cdots\times X_{|V|}$,
where each $X_{i}$ is a finite dimensional Hilbert space. For each
$i\in V$, let $f_{i}:X_{i}\to\mathbb{R}\cup\{\infty\}$ be a closed
convex function, and let $\mathbf{f}_{i}:\mathbf{X}\to\mathbb{R}\cup\{\infty\}$
be defined by $\mathbf{f}_{i}(\mathbf{x})=f_{i}([\mathbf{x}]_{i})$.
Let $\delta_{C}(\cdot)$ be the indicator function for a closed convex
set $C$. For each $\alpha\in\bar{E}$, let $H_{\alpha}\subset\mathbf{X}$
be a linear subspace, and define $\mathbf{f}_{\alpha}:\mathbf{X}\to\mathbb{R}$
by $\mathbf{f}_{\alpha}(\mathbf{x})=\delta_{H_{\alpha}}(\mathbf{x})$.
The problem we consider is 
\begin{equation}
\begin{array}{c}
\underset{\mathbf{x}\in\mathbf{X}}{\min}\frac{1}{2}\|\mathbf{x}-\bar{\mathbf{x}}\|^{2}+\underset{i\in V}{\sum}\mathbf{f}_{i}(\mathbf{x})+\underset{\alpha\in\bar{E}}{\overset{\phantom{\alpha\in\bar{E}}}{\sum}}\delta_{H_{\alpha}}(\mathbf{x}).\end{array}\label{eq:general-framework}
\end{equation}
Note that the last two sums in \eqref{eq:general-framework} can be
written as $\sum_{\alpha\in V\cup\bar{E}}\mathbf{f}_{\alpha}(\mathbf{x})$.
Typically, the hyperplanes $\{H_{\alpha}\}_{\alpha\in\bar{E}}$ are
overdetermined (see Definition \ref{def:connects} later). Partition
the set $V$ as the disjoint union \emph{$V=V_{1}\cup V_{2}\cup V_{3}\cup V_{4}$
}so that 
\begin{itemize}
\item $f_{i}(\cdot)$ are proximable functions for all $i\in V_{1}$.
\item $f_{i}(\cdot)$ are indicator functions of closed convex sets for
all $i\in V_{2}$. 
\item $f_{i}(\cdot)$ are proximable functions such that $\dom(f_{i})=X_{i}$
for all $i\in V_{3}$.
\item \emph{$f_{i}(\cdot)$ }are subdifferentiable functions (i.e., a subgradient
is easy to obtain) such that $\dom(f_{i})=X_{i}$ for all $i\in V_{4}$.
(In Sections \ref{sec:Lin-conv} and \ref{sec:O-1-k-conv-of-smooth-proximable-case},
we shall assume $f_{i}(\cdot)$ are smooth with $\nabla f_{i}(\cdot)$
is Lipschitz with modulus $\frac{1}{\sigma}$ for all $i\in V_{4}$.) 
\end{itemize}
The (Fenchel) dual of \eqref{eq:general-framework} can be found to
be 
\begin{equation}
\max_{\mathbf{z}_{\alpha}\in\mathbf{X}:\alpha\in V\cup\bar{E}}F(\{\mathbf{z}_{\alpha}\}_{\alpha\in\bar{E}\cup V}),\label{eq:general-dual}
\end{equation}
where
\begin{equation}
\begin{array}{c}
F(\{\mathbf{z}_{\alpha}\}_{\alpha\in\bar{E}\cup V}):=-\frac{1}{2}\bigg\|\bar{\mathbf{x}}-\underset{\alpha\in\bar{E}\cup V}{\sum}\mathbf{z}_{\alpha}\bigg\|^{2}+\frac{1}{2}\|\bar{\mathbf{x}}\|^{2}-\underset{\alpha\in\bar{E}\cup V}{\sum}\mathbf{f}_{\alpha}^{*}(\mathbf{z}_{\alpha}).\end{array}\label{eq:h-def}
\end{equation}
We now explain that the problem \eqref{eq:general-framework} includes
the general case of the distributed Dykstra's algorithm in \cite{Pang_Dist_Dyk,Pang_sub_Dyk}.
\begin{example}
\label{exa:distrib-dyk} \cite{Pang_Dist_Dyk,Pang_sub_Dyk}(Distributed
Dykstra's algorithm is a special case of \eqref{eq:general-framework})
Let $G=(V,E)$ be an undirected connected graph. Suppose each $X_{i}=\mathbb{R}^{m}$
for all $i\in V$, and let $\bar{E}:=E\times\{1,\dots,m\}$. For each
$\mathbf{x}\in\mathbf{X}=[\mathbb{R}^{m}]^{|V|}$, we let $[\mathbf{x}]_{i}\in\mathbb{R}^{m}$
be the $i$-th component, and we let $[[\mathbf{x}]_{i}]_{k}$ be
the $k$-th component of $[\mathbf{x}]_{i}$. For each $((i,j),k)\in\bar{E}$,
let the linear subspace $H_{((i,j),k)}\subset\mathbf{X}$ of codimension
1 be defined to be 
\begin{equation}
H_{((i,j),k)}:=\{\mathbf{x}\in\mathbf{X}:[[\mathbf{x}]_{i}]_{k}=[[\mathbf{x}]_{j}]_{k}\}.\label{eq:H-alpha-subspaces}
\end{equation}
Then the problem \eqref{eq:general-framework} is equivalent to 
\begin{equation}
\begin{array}{c}
\underset{x\in\mathbb{R}^{m}}{\min}\underset{i\in V}{\overset{\phantom{i\in V}}{\sum}}[\frac{1}{2}\|x-[\bar{\mathbf{x}}]_{i}\|^{2}+f_{i}(x)].\end{array}\label{eq:distrib-dyk-primal-pblm}
\end{equation}
\end{example}

We elaborate on distributed optimization algorithms and the features
of the distributed Dykstra's algorithm in Example \ref{exa:distrib-dyk}
in the following subsections.

\subsection{Distributed optimization }

Since this paper builds on \cite{Pang_Dist_Dyk,Pang_sub_Dyk}, we
shall give a brief introduction. Our algorithm is for the case when
the edges are undirected. But we remark on the directed case. A notable
paper based on the directed case using the subgradient algorithm is
\cite{EXTRA_Shi_Ling_Wu_Yin}, and surveys are \cite{Nedich_survey}
and \cite{Nedich_talk_2017}. The papers \cite{Nedich_Olshevsky}
and \cite{Nedich_Olshevsky_Shi} further touch on the case of time-varying
graphs. The algorithms in \cite{Bof_Carli_Schenato_2017,Vaidya_Hadji_Domin_2011}
address the averaged consensus problem for the case of directed graphs
with unreliable and reliable communications respectively. Based on
\cite{Bof_Carli_Schenato_2017}, \cite{Notarstefano_gang_Newton_2017}
uses a Newton-Raphson method to design a distributed algorithm for
directed graphs. Naturally, the communication requirements for directed
graphs need to be more stringent that the requirements for undirected
graphs. 

From here on, we discuss only algorithms for undirected graphs. A
product space formulation on the ADMM leads to a distributed algorithm
\cite[Chapter 7]{Boyd_Eckstein_ADMM_review}. Such an algorithm is
decentralized and distributed, but is not asynchronous and so can
get slowed down by slow vertices. An approach based on \cite{Eckstein_Combettes_MAPR}
allows for asynchronous operation, but is not decentralized. 

Moving beyond deterministic algorithms, distributed decentralized
asynchronous algorithms were proposed, but many of them involve some
sort of randomization. For example, the work \cite{Iutzeler_Bianchi_Ciblat_Hachem_1st_paper_dist,Bianchi_Hachem_Iutzeler_2nd_paper_dist,Wei_Ozdaglar_2013}
and the generalization \cite{AROCK_Peng_Xu_Yan_Yin} are based on
monotone operator theory (see for example the textbook \cite{BauschkeCombettes11}),
and require the computations in the nodes to follow specific probability
distributions. 

We now look at asynchronous distributed algorithm with deterministic
convergence (rather than probabilistic convergence). We mention that
incremental aggregated gradient algorithm like \cite{Gurbuzbalaban_Ozdaglar_Parrilo_SIOPT_2017,Aytekin_F_Johansson_2016}
is an algorithm for strongly convex problems that is primal in nature,
so it can't have more than one proximal term, and hence can't handle
more than one constraint set. (We consider such algorithms distinct
from what we do in this paper as they do not need the function on
each node to be strongly convex, but the algorithms need a central
node.) The method in \cite{Aybat_Hamedani_2016} is distributed and
deterministic and the averaging operation can be performed asynchronously,
but some parts of the algorithm are still required to be synchronous.

\subsection{\label{subsec:Dist-Dyk}Distributed Dykstra's algorithm }

We now recall some history of Dykstra's algorithm \cite{Dykstra83}.
Dykstra's algorithm originally solves 
\begin{equation}
\begin{array}{c}
\underset{x\in X}{\min}\frac{1}{2}\|x-\bar{x}\|^{2}+\underset{i\in V}{\overset{\phantom{i\in V}}{\sum}}\delta_{C_{i}}(x)\end{array}\label{eq:orig-Dykstra}
\end{equation}
 for closed convex sets $C_{i}$. It was also recognized to be block
coordinate minimization of the dual problem (similar to \eqref{eq:general-dual})
in \cite{Han88}, where convergence is proved under a constraint qualification
ensuring the existence of a dual minimizer. The convergence of Dykstra's
algorithm to the primal minimizer even when a dual minimizer does
not exist is sometimes known as the Boyle-Dykstra theorem \cite{BD86}.
This proof is written through duality in \cite{Gaffke_Mathar}. Other
interesting properties of Dykstra's algorithm related to this paper
are that Dykstra's algorithm converges even when the functions are
sampled in a non-cyclic order \cite{Hundal-Deutsch-97}.

Based on these previous works, we made a few extensions of Dykstra's
algorithm in \cite{Pang_Dist_Dyk} and further extended it in \cite{Pang_sub_Dyk}.
We call our algorithm the distributed Dykstra's algorithm. Its favorable
properties are:
\begin{enumerate}
\item distributed (with communications occurring only between adjacent agents
$i$ and $j$ connected by an edge). 
\item decentralized (i.e., there is no central node coordinating calculations). 
\item asynchronous (contrast this to synchronous algorithms, where the faster
agents would wait for slower agents before their next calculations). 
\item deterministic (i.e., not using any probabilistic methods, like stochastic
gradient methods).
\item able to incorporate more than one proximable function naturally. (This
largely rules out primal-only methods since they usually allow just
one proximal term.) Hence, the algorithm would be able to allow for
constrained optimization, where the feasible region is the intersection
of several sets.
\item able to allow for time-varying graphs in the sense of \cite{Nedich_Olshevsky,Nedich_Olshevsky_Shi}
(to be robust against failures of communication between two agents).
\item able to use simpler subproblems for subdifferentiable functions. 
\item able to use simpler subproblems for smooth functions.
\item able to allow for partial communication of data.
\end{enumerate}
Since Dykstra's algorithm is also dual block coordinate ascent, the
following property is obtained:
\begin{enumerate}
\item [(10)]choosing a large number of dual variables to be maximized over
gives a greedier increase of the dual objective value. 
\end{enumerate}
We are not aware of other algorithms that satisfy properties 1-5 at
the same time. We note that the approach in \cite{Aytekin_F_Johansson_2016}
is essentially a primal algorithm that allows for one proximal term
(and hence one constrained set). Due to technical difficulties (see
Remark \ref{rem:BBL99}), a dual or primal-dual method seems necessary
to handle the case of more than one constrained set. Algorithms derived
from the primal dual algorithm \cite{Chambolle_Pock_2011_PD,Chambolle_Pock_MAPR},
like \cite{Aybat_Hamedani_2016}, are very much different from what
we study in this paper. The most notable difference is that they study
ergodic convergence rates, which is not directly comparable with our
results. 

\subsubsection{Convergence rates}

Since the subproblems in our case are strongly convex, standard techniques
for block coordinate minimization, like \cite{Beck_Tetruashvili_2013,Beck_alt_min_SIOPT_2015},
can be used to prove the $O(1/k)$ convergence rate when a dual solution
exists and all functions are treated as proximable functions. 

We also showed in \cite{Pang_sub_Dyk} that if $|V|=1$, $|\bar{E}|=0$
and the function $f_{1}(\cdot)$ is subdifferentiable, then a $O(1/k)$
rate of convergence of the dual objective value can be proved with
a method somewhat similar to the bundle method, and it improves to
linear convergence if $f_{1}(\cdot)$ is smooth. (See Lemma \ref{lem:subdiff-decrease}
for more details.) The following question remains: 
\begin{equation}
\begin{array}{l}
\mbox{What is the convergence rate for the subdifferentiable distributed }\\
\mbox{Dykstra's algorithm when }|V|>1\mbox{ and }|\bar{E}|>0\mbox{?}
\end{array}\label{eq:Question}
\end{equation}

A particular case of the original Dykstra's algorithm having a linear
convergence rate is when the sets $C_{i}$ in \eqref{eq:orig-Dykstra}
are linear subspaces. It is well known that in such a case, Dykstra's
algorithm reduces to the method of alternating projections. The case
of alternating projections for the case of linear subspaces is rather
old, so we refer to the references \cite{BB96_survey,Deustch01,Deutsch01_survey,EsRa11}
for example.

\subsection{\label{subsec:Contributions}Contributions of this paper}

Given that the distributed Dykstra's algorithm has some desirable
properties as listed in Subsection \ref{subsec:Dist-Dyk}, we wish
to find out how its convergence rates compare to other well-known
distributed optimization algorithms. 

In Section \ref{sec:Lin-conv}, we prove the linear convergence (of
the dual objective function \eqref{eq:h-def}) of this method when
$V_{1}=V_{2}=V_{3}=\emptyset$ and the functions $f_{i}(\cdot)$ are
smooth for all $i\in V_{4}$ (instead of just being subdifferentiable). 

In Section \ref{sec:O-1-k-conv-of-smooth-proximable-case}, we prove
that in the case when $f_{i}(\cdot)$ is smooth for all $i\in V_{4}$
(instead of just being subdifferentiable), a dual minimizer exists,
the dual iterates are bounded, and $V_{1}$, $V_{2}$ and $V_{3}$
are not necessarily empty, the convergence rate is $O(1/k)$. This
convergence rate is the best we can expect with the distributed Dykstra's
algorithm because block coordinate minimization has a convergence
rate of at best $O(1/k)$ \cite{Beck_Tetruashvili_2013,Beck_alt_min_SIOPT_2015}. 

In Section \ref{sec:Sublin-conv}, we establish a $O(1/k^{1/3})$
convergence rate for the distributed Dykstra's algorithm in the general
case when a dual minimizer exists and the dual iterates are bounded,
addressing the question in \eqref{eq:Question}. When there are no
subdifferentiable functions, the common rate of $O(1/k)$ is obtained.
While the $O(1/k^{1/3})$ rate is slower than the subgradient algorithm,
our experimental results suggest a $O(1/k)$ rate for our set of problems.
And as mentioned, we are not aware of any other distributed optimization
algorithms with properties (1) to (5). Our algorithm is also not easily
comparable to the subgradient algorithm because the subgradient algorithm
does not include problems whose domain is the intersection of more
than one convex set (see the issues for such problems in Remark \ref{rem:BBL99}).
We hope our work can lead to subsequent research for distributed problems.

\subsection{Notation}

The functions $\mathbf{f}_{\alpha}(\cdot)$ are indexed by $\alpha\in V\cup\bar{E}$,
and sometimes $\beta\in V\cup\bar{E}$. We use $\mathbf{f}_{i}(\cdot)$
when we want to index with $i\in V$, and we sometimes use $j\in V$.
We usually reserve bold variables like $\mathbf{x}$, $\mathbf{z}_{\alpha}$
and $\mathbf{s}_{\alpha}$ to be variables in $\mathbf{X}$, but we
also use $\mathbf{z}$, $\mathbf{s}$ for vectors in $\mathbf{X}^{|V\cup\bar{E}|}$.
We use $[\mathbf{x}]_{i}\in X_{i}$ in order to index the $i$-th
component of $\mathbf{x}\in\mathbf{X}$. We usually use $x$, $s$
to represent vectors in $X_{i}$. 

\section{Algorithm statement and preliminaries }

In this section, we list down the preliminaries and description of
the distributed Dykstra's algorithm studied in \cite{Pang_Dist_Dyk,Pang_sub_Dyk}.
We do not claim originality in this section, and we recall some results
useful for subsequent proofs.

For all $n\geq1$ and $w\in\{1,\dots,\bar{w}\}$, define $\mathbf{f}_{\alpha,n,w}:\mathbf{X}\to\mathbb{R}$
by \begin{subequations}\label{eq_m:h-a-n-w} 
\begin{eqnarray}
\mathbf{f}_{\alpha,n,w}(\cdot) & = & \mathbf{f}_{\alpha}(\cdot)\mbox{ for all }\alpha\in[\bar{E}\cup V]\backslash V_{4}\label{eq:h-a-n-w-eq-h-a}\\
\mbox{ and }\mathbf{f}_{\alpha,n,w}(\cdot) & \leq & \mathbf{f}_{\alpha}(\cdot)\mbox{ for all }\alpha\in V_{4}.\label{eq:h-a-n-w-lesser}
\end{eqnarray}
\end{subequations}Define the function $F^{n,w}:\mathbf{X}^{|V\cup\bar{E}|}\to\mathbb{R}\cup\{\infty\}$
to be 
\begin{equation}
\begin{array}{c}
F^{n,w}(\{\mathbf{z}_{\alpha}\}_{\alpha\in\bar{E}\cup V}):=-\frac{1}{2}\bigg\|\bar{\mathbf{x}}-\underset{\alpha\in\bar{E}\cup V}{\sum}\mathbf{z}_{\alpha}\bigg\|^{2}+\frac{1}{2}\|\bar{\mathbf{x}}\|^{2}-\underset{\alpha\in\bar{E}\cup V}{\sum}\mathbf{f}_{\alpha,n,w}^{*}(\mathbf{z}_{\alpha}).\end{array}\label{eq:h-def-1}
\end{equation}
Based on our original motivation in Example \ref{exa:distrib-dyk}
from \cite{Pang_Dist_Dyk,Pang_sub_Dyk}, we make the following definition. 
\begin{defn}
\label{def:connects}Let $D:=\cap_{\alpha\in\bar{E}}H_{\alpha}$.
We say that a subset $E'\subset\bar{E}$ \emph{connects} $V$ if 
\begin{equation}
\cap_{\alpha\in E'}H_{\alpha}=D.\label{eq:intersect-H}
\end{equation}
\end{defn}

Since $H_{\alpha}$ were assumed to be linear subspaces, it is clear
that condition \eqref{eq:intersect-H} on $E'$ is equivalent to 
\begin{equation}
\begin{array}{c}
\underset{\alpha\in E'}{\overset{\phantom{\alpha\in E'}}{\sum}}H_{\alpha}^{\perp}=D^{\perp}.\end{array}\label{eq:D-perp}
\end{equation}
The following simple result does not play much of a role here. But
it plays a role in \cite{Pang_Dist_Dyk,Pang_sub_Dyk} to show the
convergence for time-varying graphs when a dual minimizer may not
exist. This explains line 5 in Algorithm \ref{alg:Ext-Dyk}.
\begin{lem}
\label{lem:express-v}There is a constant $C_{reg}>0$ such that for
any $\mathbf{v}\in D^{\perp}$ and any $E'\subset\bar{E}$ such that
$E'$ connects $V$, we can write $\mathbf{v}=\sum_{\alpha\in E'}\mathbf{v}_{\alpha}$
so that\textbf{ $\mathbf{v}_{\alpha}\in H_{\alpha}^{\perp}$ }and
$\|\mathbf{v}_{\alpha}\|\leq C_{reg}\|\mathbf{v}\|$ for all $\alpha\in E'$. 
\end{lem}

\begin{proof}
The proof of a slightly weaker result was in \cite{Pang_Dist_Dyk}
to accommodate the case when $\mathbf{X}$ is not necessarily $[\mathbb{R}^{m}]^{|V|}$,
but we include this proof for completeness. Since we have assumed
finite dimensionality of all $X_{i}$'s, for all $\alpha\in\bar{E}$,
we can write $H_{\alpha}^{\perp}\subset\mathbf{X}$ as $H_{\alpha}^{\perp}:=\spn\{\mathbf{v}_{\alpha,1},\dots,\mathbf{v}_{\alpha,d_{\alpha}}\}$,
where $\{\mathbf{v}_{\alpha,1},\dots,\mathbf{v}_{\alpha,d_{\alpha}}\}$
is a set of $d_{\alpha}$ linearly independent vectors and $d_{\alpha}$
is the dimension of $H_{\alpha}^{\perp}$. For all $E'\subset\bar{E}$,
we can find a subset of the $S':=\cup_{\alpha\in E'}\cup_{i=1}^{d_{\alpha}}\{\mathbf{v}_{\alpha,i}\}$
so that $S'$ is a set of linearly independent vectors that span $D^{\perp}$.
From basic linear algebra, we can find a constant $C_{E'}$ such that
for any $\mathbf{v}\in\bar{H}^{\perp}$, we can write $\mathbf{v}$
as a linear combination of the vectors in $S'$ so that each component
has norm less than $C_{E'}\|\mathbf{v}\|$. The result follows readily. 
\end{proof}
To simplify calculations, we let the vectors $\mathbf{v}_{A}$, $\mathbf{v}_{H}$
and $\mathbf{x}$ in $\mathbf{X}$ be denoted by\begin{subequations}\label{eq_m:all_acronyms}
\begin{eqnarray}
\mathbf{v}_{H} & = & \begin{array}{c}
\underset{\alpha\in\bar{E}}{\overset{\phantom{i\in V}}{\sum}}\mathbf{z}_{\alpha}\end{array}\label{eq:v-H-def}\\
\mathbf{v}_{A} & = & \begin{array}{c}
\mathbf{v}_{H}+\underset{i\in V}{\overset{\phantom{i\in V}}{\sum}}\mathbf{z}_{i}\end{array}\label{eq:from-10}\\
\mathbf{x} & = & \begin{array}{c}
\bar{\mathbf{x}}-\mathbf{v}_{A}.\end{array}\label{eq:x-from-v-A}
\end{eqnarray}
\end{subequations}

We now state Algorithm \vref{alg:Ext-Dyk}. Algorithm \ref{alg:Ext-Dyk}
calls on Algorithm \vref{alg:subdiff-subalg} as a subalgorithm. 

\begin{algorithm}[!h]
\begin{lyxalgorithm}
\label{alg:Ext-Dyk}(Distributed Dykstra's algorithm) Consider the
problem \eqref{eq:general-framework} along with the associated dual
problem \eqref{eq:general-dual}.

Let $\bar{w}$ be a positive integer. Let $C_{reg}>0$ satisfy Lemma
\ref{lem:express-v}. For each $\alpha\in[\bar{E}\cup V]\backslash V_{4}$,
$n\geq1$ and $w\in\{1,\dots,\bar{w}\}$, let $\mathbf{f}_{\alpha,n,w}:\mathbf{X}\to\mathbb{R}$
be as defined in \eqref{eq_m:h-a-n-w}. Our distributed Dykstra's
algorithm is as follows:

01$\quad$Let 

\begin{itemize}
\item $\mathbf{z}_{i}^{1,0}\in\mathbf{X}$ be a starting dual vector for
$\mathbf{f}_{i}(\cdot)$ for each $i\in V$ so that $[\mathbf{z}_{i}^{1,0}]_{j}=0\in X_{j}$
for all $j\in V\backslash\{i\}$. 
\item $\mathbf{v}_{H}^{1,0}\in D^{\perp}$ be a starting dual vector.

\begin{itemize}
\item Note: $\{\mathbf{z}_{\alpha}^{n,0}\}_{\alpha\in\bar{E}}$ is defined
through $\mathbf{v}_{H}^{n,0}$ in \eqref{eq_m:resetted-z-i-j}.
\end{itemize}
\item Let $\mathbf{x}^{1,0}$ be $\mathbf{x}^{1,0}=\bar{\mathbf{x}}-\mathbf{v}_{H}^{1,0}-\sum_{i\in V}\mathbf{z}_{i}^{1,0}$.
\end{itemize}
02$\quad$For each $i\in V_{4}$, let $\mathbf{f}_{i,1,0}:\mathbf{X}\to\mathbb{R}$
be a function such that $\mathbf{f}_{i,1,0}(\cdot)\leq\mathbf{f}_{i}(\cdot)$

03$\quad$For $n=1,2,\dots$

04$\quad$$\quad$\textup{Let $\bar{E}_{n}\subset\bar{E}$ be such
that $\bar{E}_{n}$ connects $V$ in the sense of Definition \ref{def:connects}.}

05$\quad$$\quad$\textup{Define $\{\mathbf{z}_{\alpha}^{n,0}\}_{\alpha\in\bar{E}}$
so that:\begin{subequations}\label{eq_m:resetted-z-i-j} 
\begin{eqnarray}
\mathbf{z}_{\alpha}^{n,0} & = & 0\mbox{ for all }\alpha\notin\bar{E}_{n}\label{eq:reset-z-i-j-1}\\
\mathbf{z}_{\alpha}^{n,0} & \in & H_{\alpha}^{\perp}\mbox{ for all }\alpha\in\bar{E}\label{eq:reset-z-i-j-2}\\
\|\mathbf{z}_{\alpha}^{n,0}\| & \leq & C_{reg}\|\mathbf{v}_{H}^{n,0}\|\mbox{ for all }\alpha\in\bar{E}\label{eq:reset-z-i-j-3}\\
\mbox{ and }\sum_{\alpha\in\bar{E}}\mathbf{z}_{\alpha}^{n,0} & = & \mathbf{v}_{H}^{n,0}.\label{eq:reset-z-i-j-4}
\end{eqnarray}
\end{subequations}}

$\quad$$\quad$\textup{(This is possible by Lemma \ref{lem:express-v}.) }

06$\quad$$\quad$For $w=1,2,\dots,\bar{w}$

07$\quad$$\quad$$\quad$Choose a set $S_{n,w}\subset\bar{E}_{n}\cup V$
such that $S_{n,w}\neq\emptyset$. 

08$\quad$$\quad$$\quad$If $S_{n,w}\subset V_{4}$, then 

09$\quad$$\quad$$\quad$$\quad$Apply Algorithm \ref{alg:subdiff-subalg}.

10$\quad$$\quad$$\quad$else

11$\quad$$\quad$$\quad$$\quad$Set $\mathbf{f}_{i,n,w}(\cdot):=\mathbf{f}_{i,n,w-1}(\cdot)$
for all $i\in V_{4}$.

12$\quad$$\quad$$\quad$$\quad$Define $\{\mathbf{z}_{\alpha}^{n,w}\}_{\alpha\in S_{n,w}}$
by 
\begin{equation}
\begin{array}{c}
\{\mathbf{z}_{\alpha}^{n,w}\}_{\alpha\in S_{n,w}}=\underset{\mathbf{z}_{\alpha},\alpha\in S_{n,w}}{\arg\min}\frac{1}{2}\Big\|\bar{\mathbf{x}}-\underset{\alpha\notin S_{n,w}}{\overset{\phantom{\alpha\notin S_{n,w}}}{\sum}}\mathbf{z}_{\alpha}^{n,w-1}-\underset{\alpha\in S_{n,w}}{\sum}\mathbf{z}_{\alpha}\Big\|^{2}+\underset{\alpha\in S_{n,w}}{\overset{\phantom{\alpha\in S_{n,w}}}{\sum}}\mathbf{f}_{\alpha,n,w}^{*}(\mathbf{z}_{\alpha}).\end{array}\label{eq:Dykstra-min-subpblm}
\end{equation}

13$\quad$$\quad$$\quad$end if 

14$\quad$$\quad$$\quad$Set $\mathbf{z}_{\alpha}^{n,w}:=\mathbf{z}_{\alpha}^{n,w-1}$
for all $\alpha\notin S_{n,w}$.

15$\quad$$\quad$End For 

16$\quad$$\quad$Let $\mathbf{z}_{i}^{n+1,0}=\mathbf{z}_{i}^{n,\bar{w}}$
for all $i\in V$ and $\mathbf{v}_{H}^{n+1,0}=\mathbf{v}_{H}^{n,\bar{w}}=\sum_{\alpha\in\bar{E}}\mathbf{z}_{\alpha}^{n,\bar{w}}$.

17$\quad$$\quad$Let $\mathbf{f}_{i,n+1,0}(\cdot)=\mathbf{f}_{i,n,\bar{w}}(\cdot)$
for all $i\in V_{4}$.

18$\quad$End For 
\end{lyxalgorithm}

\end{algorithm}

\begin{algorithm}[!h]
\begin{lyxalgorithm}
\label{alg:subdiff-subalg}(Subalgorithm for subdifferentiable functions)
This algorithm is run when line 9 of Algorithm \ref{alg:Ext-Dyk}
is reached. Note that to get to this subalgorithm, $S_{n,w}\subset V_{4}$.
Suppose Assumption \ref{assu:to-start-subalg} holds.

01 For each $i\in S_{n,w}$ 

02 $\quad$For $\tilde{f}_{i,n,w-1}:X_{i}\to\mathbb{R}$ defined by
\begin{equation}
\tilde{f}_{i,n,w-1}(x):=f_{i}([\bar{\mathbf{x}}-\mathbf{v}_{H}^{n,w-1}-\mathbf{z}_{i}^{n,w-1}]_{i})+\langle s,x-[\bar{\mathbf{x}}-\mathbf{v}_{H}^{n,w-1}-\mathbf{z}_{i}^{n,w-1}]_{i}\rangle,\label{eq:linearize-f-i-n-w}
\end{equation}

where $s\in\partial f_{i}([\bar{\mathbf{x}}-\mathbf{v}_{H}^{n,w-1}-\mathbf{z}_{i}^{n,w-1}]_{i})$,
consider 
\begin{equation}
\begin{array}{c}
\underset{x\in X_{i}}{\min}\left[\frac{1}{2}\|[\bar{\mathbf{x}}-\mathbf{v}_{H}^{n,w-1}]_{i}-x\|^{2}+\max\{f_{i,n,w-1},\tilde{f}_{i,n,w-1}\}(x)\right],\end{array}\label{eq:alg-primal-subpblm}
\end{equation}

03 $\quad$Let the primal and dual solutions of \eqref{eq:alg-primal-subpblm}
be $x_{i}^{+}$ and $z_{i}^{+}$ 

04 $\quad$Define $f_{i,n,w}:X_{i}\to\mathbb{R}$ to be the affine
function 
\begin{equation}
f_{i,n,w}(x):=f_{i,n,w-1}(x_{i}^{+})+\langle x-x_{i}^{+},[\bar{\mathbf{x}}-\mathbf{v}_{H}^{n,w-1}]_{i}-x_{i}^{+}\rangle.\label{eq:def-f-i-n-w}
\end{equation}

05 $\quad$In other words, $f_{i,n,w}(\cdot)$ is chosen such that
the 

$\qquad\qquad$primal and dual optimizers to \eqref{eq:alg-primal-subpblm}
coincide with that of 
\begin{equation}
\begin{array}{c}
\underset{x\in X_{i}}{\min}\left[\frac{1}{2}\|[\bar{\mathbf{x}}-\mathbf{v}_{H}^{n,w-1}]_{i}-x\|^{2}+f_{i,n,w}(x)\right].\end{array}\label{eq:finw-design}
\end{equation}

06 $\quad$Define the function $\mathbf{f}_{i,n,w}:\mathbf{X}\to\mathbb{R}$
and the dual vector $\mathbf{z}_{i}^{n,w}\in\mathbf{X}$ to be 
\begin{equation}
\mathbf{f}_{i,n,w}(\mathbf{x}):=f_{i,n,w}([\mathbf{x}]_{i})\mbox{ and }[\mathbf{z}_{i}^{n,w}]_{j}:=\begin{cases}
z_{i}^{+} & \mbox{ if }j=i\\
0 & \mbox{ if }j\neq i.
\end{cases}\label{eq:def-z-i}
\end{equation}

07 End for 

08 For all $i\in V_{4}\backslash S_{n,w}$, $\mathbf{f}_{i,n,w}(\cdot)=\mathbf{f}_{i,n,w-1}(\cdot)$.
\end{lyxalgorithm}

\end{algorithm}

\begin{rem}
(Intuition behind Algorithms \ref{alg:Ext-Dyk} and \ref{alg:subdiff-subalg})
(Intuition behind Algorithms \ref{alg:Ext-Dyk} and \ref{alg:subdiff-subalg})
We summarize the intuition behind Algorithms \ref{alg:Ext-Dyk} and
\ref{alg:subdiff-subalg}. Dykstra's algorithm is block coordinate
ascent on the dual \eqref{eq:general-dual}, and this is reflected
in lines 7-14 of Algorithm \ref{alg:Ext-Dyk}. That is, find $\mathbf{z}\in\mathbf{X}^{|\bar{E}\cup V|}$
that tries to improve the objective value of \eqref{eq:h-def}. As
explained in \cite{Pang_Dist_Dyk}, one only needs to keep track of
$\mathbf{x}_{i}$ and $[\mathbf{z}_{i}]_{i}$ for all $i\in V$, and
not all the variables. Line 5 corrects $\{\mathbf{z}_{\alpha}\}_{\alpha\in\bar{E}}$
so that the dual objective value remains the same, and this consideration
is needed when we try to prove that the algorithm works for time-varying
graphs. Lastly, to explain Algorithm \ref{alg:subdiff-subalg}, note
that if $\mathbf{f}_{i}(\cdot)$ is subdifferentiable, then we can
form\textbf{ }affine minorants $\tilde{\mathbf{f}}_{i}(\cdot)$ of
$\mathbf{f}_{i}(\cdot)$ so that $\tilde{\mathbf{f}}_{i}(\cdot)\leq\mathbf{f}_{i}(\cdot)$
and $\tilde{\mathbf{f}}_{i}^{*}(\cdot)\geq\mathbf{f}_{i}^{*}(\cdot)$.
This process is similar to the bundle method. By iteratively updating
$\tilde{\mathbf{f}}_{i}^{*}(\cdot)$ and maximizing 
\[
\begin{array}{c}
-\frac{1}{2}\Big\|\bar{\mathbf{x}}-\underset{\alpha\in\bar{E}_{n}\cup V}{\sum}\mathbf{z}_{\alpha}\Big\|^{2}+\frac{1}{2}\|\bar{\mathbf{x}}\|^{2}-\underset{\alpha\in\bar{E}_{n}}{\sum}\delta_{H_{\alpha}}^{*}(\mathbf{z}_{\alpha})-\underset{i\in V}{\sum}\tilde{\mathbf{f}}_{i}^{*}(\mathbf{z}_{i}),\end{array}
\]
we can converge to the optimal dual objective value. 
\end{rem}

The following result is essential for showing that the distributed
Dykstra's algorithm is asynchronous, and will be useful in some proofs
in this paper.
\begin{prop}
\label{prop:sparsity}(Sparsity of $z_{\alpha}$) We have $[\mathbf{z}_{i}^{n,w}]_{j}=0$
for all $j\in V\backslash\{i\}$, $n\geq1$ and $w\in\{0,1,\dots,\bar{w}\}$.
Furthermore, in the case of Example \ref{exa:distrib-dyk} where $X_{i}=\mathbb{R}^{m}$
for all $i\in V$ and $H_{(e,k)}$ are defined as in \eqref{eq:H-alpha-subspaces}
for all $(e,k)\in E\times\{1,\dots,m\}$, the vector $\mathbf{z}_{(e,k)}^{n,w}\in[\mathbb{R}^{m}]^{|V|}$
satisfies $[[\mathbf{z}_{(e,k)}^{n,w}]_{i'}]_{k'}=0$ unless $k=k'$
and $i$ is an endpoint of $e$.
\end{prop}

\begin{proof}
[Sketch of proof]The proof of this result is similar to the corresponding
result in \cite{Pang_Dist_Dyk}. The claim for $\mathbf{z}_{i}^{n,w}$
relies on the fact that $\mathbf{f}_{i,n,w}(\cdot)$ depends only
on the $i$-th component, and the claim for $\mathbf{z}_{(e,k)}^{n,w}$
relies on the fact that $\mathbf{f}_{(e,k)}(\cdot)=\delta_{H_{(e,k)}^{\perp}}(\cdot)$,
with $H_{(e,k)}^{\perp}$ containing vectors that are zero in all
but 2 coordinates.
\end{proof}
We will come back to the setting of Example \ref{exa:distrib-dyk}
to prove specific results. Another fact we will use later is that
under the setting in Example \ref{exa:distrib-dyk}, the set $D$
defined through \eqref{eq:intersect-H}  has the simplifications\begin{subequations}
\begin{eqnarray}
 &  & \begin{array}{c}
D^{\phantom{\perp}}=\{\mathbf{x}\in[\mathbb{R}^{m}]^{|V|}:\mathbf{x}=(x,x,\dots,x)\mbox{ for some }x\in\mathbb{R}^{m}\}\end{array}\label{eq:D-formula-a}\\
 & \mbox{ and } & \begin{array}{c}
D^{\perp}=\Big\{\mathbf{x}\in[\mathbb{R}^{m}]^{|V|}:\underset{i\in V}{\overset{\phantom{i\in V}}{\sum}}[\mathbf{x}]_{i}=0\Big\}.\end{array}\label{eq:D-formula-b}
\end{eqnarray}
\end{subequations}We state some notation necessary for further discussions.
For any $\alpha\in\bar{E}\cup V$ and $n\in\{1,2,\dots\}$, let $p(n,\alpha)$
be 
\begin{equation}
p(n,\alpha):=\max\{w':w'\leq\bar{w},\alpha\in S_{n,w'}\}.\label{eq:p-n-alpha}
\end{equation}
In other words, $p(n,\alpha)$ is the index $w'$ such that $\alpha\in S_{n,w'}$
but $\alpha\notin S_{n,k}$ for all $k\in\{w'+1,\dots,\bar{w}\}$.
We make three assumptions listed below.
\begin{assumption}
\label{assu:to-start-subalg}(Start of Algorithm \ref{alg:subdiff-subalg})
Recall that at the start of Algorithm \ref{alg:subdiff-subalg}, $S_{n,w}\subset V_{4}$.
We make three assumptions.

\begin{enumerate}
\item Whenever $(n,w)$ is such that $w>1$ and $S_{n,w}\subset V_{4}$
so that Algorithm \ref{alg:subdiff-subalg} is invoked, each $\mathbf{z}_{i}^{n,w-1}\in\mathbf{X}$,
where $i\in V_{4}$, is such that $[\mathbf{z}_{i}^{n,w-1}]_{i}\in X_{i}$
is the optimizer to the problem 
\begin{equation}
\begin{array}{c}
\underset{z\in X_{i}}{\min}\frac{1}{2}\|[\bar{\mathbf{x}}-\mathbf{v}_{H}^{n,w-1}]_{i}-z\|^{2}+f_{i,n,w-1}^{*}(z).\end{array}\label{eq:multi-node-start}
\end{equation}
In other words, suppose $w_{i}\geq1$ is the largest $w'$ such that
$i\in S_{n,w'}$ and $i\notin S_{n,\tilde{w}}$ for all $\tilde{w}\in\{w'+1,w'+2,\dots,w-1\}$.
Then for all $\tilde{w}\in\{w_{i}+1,\dots,w-1\}$, and $\alpha\in S_{n,\tilde{w}}$,
the condition $\mathbf{v}\in H_{\alpha}$ implies $[\mathbf{v}]_{i}=0$.
\item Suppose that for all $i\in V_{4}$, $\tilde{w}\in\{p(n,i)+1,\dots,\bar{w}\}$
and $\alpha\in S_{n,\tilde{w}}$, the condition $\mathbf{v}\in H_{\alpha}$
implies $[\mathbf{v}]_{i}=0$. (This implies $\mathbf{x}_{i}^{n,p(n,i)}=\mathbf{x}_{i}^{n,\bar{w}}$.)
\item Suppose that $S_{n,1}=V_{4}$ for all $n>1$.
\end{enumerate}
\end{assumption}

With these assumptions, we are able to prove the following. Even though
the proof in \cite{Pang_sub_Dyk} for the analogue of Theorem \ref{thm:convergence}
below was for the case of Example \ref{exa:distrib-dyk}, the proofs
can be carried over in a straightforward manner. 
\begin{thm}
\label{thm:convergence} \cite{Pang_sub_Dyk}(Convergence to primal
minimizer) Consider Algorithm \ref{alg:Ext-Dyk}. Assume that the
problem \eqref{eq:general-framework} is feasible, and for all $n\geq1$,
$\bar{E}_{n}=[\cup_{w=1}^{\bar{w}}S_{n,w}]\cap\bar{E}$, and $[\cup_{w=1}^{\bar{w}}S_{n,w}]\supset V$.
Suppose also that Assumption \ref{assu:to-start-subalg} holds.

For the sequence $\{\mathbf{z}_{\alpha}^{n,w}\}_{{1\leq n<\infty\atop 0\leq w\leq\bar{w}}}\subset\mathbf{X}$
for each $\alpha\in\bar{E}\cup V$ generated by Algorithm \ref{alg:Ext-Dyk}
and the sequences $\{v_{H}^{n,w}\}_{{1\leq n<\infty\atop 0\leq w\leq\bar{w}}}\subset\mathbf{X}$
and $\{v_{A}^{n,w}\}_{{1\leq n<\infty\atop 0\leq w\leq\bar{w}}}\subset\mathbf{X}$
thus derived, we have:

\begin{enumerate}
\item [(i)]For all $n\geq1$ and $w_{1},w_{2}\in\{1,\dots,\bar{w}\}$ such
that $w_{1}\leq w_{2}$, 
\[
\begin{array}{c}
F^{n,w_{2}}(\mathbf{z}^{n,w_{2}})\geq F^{n,w_{1}}(\mathbf{z}^{n,w_{1}})+\frac{1}{2}\underset{w'=w_{1}+1}{\overset{w_{2}}{\sum}}\|\mathbf{v}_{A}^{n,w'}-\mathbf{v}_{A}^{n,w'-1}\|^{2}.\end{array}
\]
Hence the sum $\sum_{n=1}^{\infty}\sum_{w=1}^{\bar{w}}\|\mathbf{v}_{A}^{n,w}-\mathbf{v}_{A}^{n,w-1}\|^{2}$
is finite and $\{F^{n,\bar{w}}(\{\mathbf{z}_{\alpha}^{n,\bar{w}}\}_{\alpha\in\bar{E}\cup V})\}_{n=1}^{\infty}$
is nondecreasing.
\item [(ii)]There is a constant $C$ such that $\|\mathbf{v}_{A}^{n,w}\|^{2}\leq C$
for all $n\in\mathbb{N}$ and $w\in\{1,\dots,\bar{w}\}$. 
\item [(iii)]For all $i\in V_{3}\cup V_{4}$, $n\geq1$ and $w\in\{1,\dots,\bar{w}\}$,
the vectors $\mathbf{z}_{i}^{n,w}$ are bounded. 
\end{enumerate}
\end{thm}

We now list down a result that will be useful for showing the decrease
of the dual objective value in terms of $f_{i}([\mathbf{x}^{n,\bar{w}}]_{i})-f_{i,n,\bar{w}}([\mathbf{x}^{n,\bar{w}}]_{i})$. 
\begin{lem}
\label{lem:subdiff-decrease}\cite{Pang_sub_Dyk} Suppose $f:X\to\mathbb{R}$
is a closed convex subdifferentiable function such that $\dom(f)=X$.
Consider the problem 
\begin{equation}
\begin{array}{c}
\underset{x}{\overset{\phantom{x}}{\min}}\,\,f(x)+\frac{1}{2}\|x-\bar{x}\|^{2},\end{array}\label{eq:lemma-primal}
\end{equation}
which has (Fenchel) dual 
\begin{equation}
\begin{array}{c}
\underset{x}{\overset{\phantom{x}}{\max}}\,\,-f^{*}(z)+\frac{1}{2}\|\bar{x}\|-\frac{1}{2}\|z-\bar{x}\|^{2}.\end{array}\label{eq:lemma-dual}
\end{equation}
Strong duality is satisfied for this primal dual pair. Let the common
objective value be $v^{*}$. Let $f_{1}:X\to\mathbb{R}$ be an affine
function $f_{1}(x):=a_{1}^{T}x+b_{1}$ such that $f_{1}(\cdot)\leq f(\cdot)$.
We have $f_{1}^{*}(\cdot)\geq f^{*}(\cdot)$. Let $z_{1}$ be the
maximizer of $\max_{z}-f_{1}^{*}(z)+\frac{1}{2}\|\bar{x}\|^{2}-\frac{1}{2}\|z-\bar{x}\|^{2}$,
and let the corresponding solution to the primal problem $\min_{x}f_{1}(x)+\frac{1}{2}\|x-\bar{x}\|^{2}$
be $x_{1}$. Define $\tilde{f}_{1}:X\to\mathbb{R}$ to be an affine
minorant of $f(\cdot)$ at $x_{1}$, i.e., $\tilde{f}_{1}(x)=f(x)+s_{1}^{T}(x-x_{1})$
for some $s_{1}\in\partial f(x_{1})$. Let $x_{2}$ be the minimizer
to the problem 
\begin{equation}
\begin{array}{c}
\underset{x}{\overset{\phantom{x}}{\min}}\,\,[\max\{f_{1}(x),\tilde{f}_{1}(x)\}+\frac{1}{2}\|x-\bar{x}\|^{2}],\end{array}\label{eq:min-max-pblm}
\end{equation}
and let $z_{2}$ be the dual solution. Let $f_{2}:X\to\mathbb{R}$
be the affine function such that the problem \textup{
\[
\begin{array}{c}
\underset{x}{\overset{\phantom{x}}{\min}}f_{2}(x)+\frac{1}{2}\|x-\bar{x}\|^{2}\end{array}
\]
}has the same primal and dual solutions $x_{2}$ and $z_{2}$. Let
\[
\begin{array}{c}
\alpha_{i}=v^{*}-[-f_{i}^{*}(z_{i})+\frac{1}{2}\|\bar{x}\|^{2}-\frac{1}{2}\|z_{i}-\bar{x}\|^{2}]\mbox{ for }i=1,2.\end{array}
\]
One can see that $\alpha_{i}\geq0$, and $\alpha_{i}$ is the measure
of the gap between the estimate of the dual objective value \eqref{eq:lemma-dual}
and its true value $v^{*}$. We have the following:

\begin{enumerate}
\item Let $L$ be the Lipschitz constant of $f(\cdot)$. Then
\begin{equation}
\begin{array}{c}
\begin{array}{c}
\alpha_{2}\leq\alpha_{1}-\frac{1}{2}t^{2},\mbox{ where }\frac{1}{2(L+1)^{2}}t^{2}+t\geq[f(x_{1})-f_{1}(x_{1})].\end{array}\end{array}\label{eq:target-quad}
\end{equation}
\item If $f(\cdot)$ is smooth and $\nabla f(\cdot)$ is Lipschitz with
constant $L'$, then 
\[
\begin{array}{c}
\frac{1}{4(L'+1)}\left(\frac{\alpha_{2}}{\alpha_{1}}\right)^{2}+\frac{\alpha_{2}}{\alpha_{1}}\leq1.\end{array}
\]
\end{enumerate}
\end{lem}

Note that $f(x_{1})-f_{1}(x_{1})$ in \eqref{eq:target-quad} plays
the role of $f_{i}([\mathbf{x}^{n,\bar{w}}]_{i})-f_{i,n,\bar{w}}([\mathbf{x}^{n,\bar{w}}]_{i})$
in the proofs in Section \ref{sec:Sublin-conv}.

Recall the definition of $p(\cdot,\cdot)$ in \eqref{eq:p-n-alpha}.
It follows from line 14 in Algorithm \ref{alg:Ext-Dyk} that 
\begin{equation}
\mathbf{z}_{\alpha}^{n,p(n,\alpha)}=\mathbf{z}_{\alpha}^{n,p(n,\alpha)+1}=\cdots=\mathbf{z}_{\alpha}^{n,\bar{w}}\mbox{ for all }\alpha\in\bar{E}\cup V.\label{eq:stagnant-indices}
\end{equation}
Moreover, $\alpha\notin\bar{E}_{n}$ implies $\alpha\notin S_{n,w}$
for all $w\in\{1,\dots,\bar{w}\}$, so 
\begin{equation}
0\overset{\scriptsize\eqref{eq:reset-z-i-j-1}}{=}\mathbf{z}_{\alpha}^{n,0}=\mathbf{z}_{\alpha}^{n,1}=\cdots=\mathbf{z}_{\alpha}^{n,\bar{w}}\mbox{ for all }\alpha\in\bar{E}\backslash\bar{E}_{n}.\label{eq:zero-indices}
\end{equation}

\begin{rem}
(On the condition $S_{n,1}=V_{4}$) Throughout this paper, we assumed
$S_{n,1}=V_{4}$ in Assumption \ref{assu:to-start-subalg}. Algorithm
\ref{alg:Ext-Dyk} with this condition would not be truly asynchronous,
but it is relatively easy to enforce this condition. One way to enforce
this condition is to use a global clock. Another way to enforce this
condition is to use the sparsity of $\mathbf{z}_{\alpha}$ in Proposition
\ref{prop:sparsity}. We limit ourselves to the special case in Example
\ref{exa:distrib-dyk}. Suppose that $\{S_{n,w}\}_{w=1}^{\bar{w}}$
is such that for all $i\in V_{4}$, $S_{n,w_{i}}=\{i\}$ for some
$w_{i}\in\{1,\dots,\bar{w}\}$. Suppose also that for all $i,j\in V_{4}$
such that $w_{i}<w_{j}$:

\begin{itemize}
\item [($\star$)]There are no $(e,k)\in\bar{E}$ such that $i$ and $j$
are the two endpoints of $e$ and $(e,k)\in S_{n,w'}$ for some $w'$
such that $w_{i}<w'<w_{j}$. 
\end{itemize}
If condition ($\star$) holds for some $i,j\in V_{4}$, then the sparsity
of $\mathbf{z}_{\alpha}^{n,w}$ implies that if we changed from $S_{n,w_{i}}=\{i\}$
and $S_{n,w_{j}}=\{j\}$ to $S_{n,w_{i}}=\{i,j\}$ and $S_{n,w_{j}}=\emptyset$,
then the iterates $\{\mathbf{x}^{n,w}\}_{w}$ obtained will remain
equivalent. It is possible to ensure ($\star$) for all $i,j\in V_{4}$
using a signal from a fixed node in $V$ propagated as computations
in the algorithm are carried out. 

\begin{rem}
(The dual objective value) As we have discussed in \cite{Pang_Dist_Dyk,Pang_sub_Dyk},
strong duality holds for the problem \eqref{eq:general-framework}.
In particular, for any primal feasible $\mathbf{x}$ and dual feasible
$\mathbf{z}$, the duality gap satisfies 
\begin{eqnarray}
 &  & \begin{array}{c}
\frac{1}{2}\|\bar{\mathbf{x}}-\mathbf{x}\|^{2}+\underset{\alpha\in\bar{E}\cup V}{\sum}\mathbf{f}_{\alpha}(\mathbf{x})-F(\{\mathbf{z}_{\alpha}\}_{\alpha\in\bar{E}\cup V})\end{array}\label{eq:From-8}\\
 & \overset{\eqref{eq:h-def}}{=} & \begin{array}{c}
\frac{1}{2}\|\bar{\mathbf{x}}-\mathbf{x}\|^{2}+\underset{\alpha\in\bar{E}\cup V}{\sum}[\mathbf{f}_{\alpha}(\mathbf{x})+\mathbf{f}_{\alpha}^{*}(\mathbf{z}_{\alpha})]-\left\langle \bar{\mathbf{x}},\underset{\alpha\in\bar{E}\cup V}{\sum}\mathbf{z}_{\alpha}\right\rangle +\frac{1}{2}\left\Vert \underset{\alpha\in\bar{E}\cup V}{\sum}\mathbf{z}_{\alpha}\right\Vert ^{2}\end{array}\nonumber \\
 & \overset{\scriptsize\mbox{Fenchel duality}}{\geq} & \begin{array}{c}
\frac{1}{2}\|\bar{\mathbf{x}}-\mathbf{x}\|^{2}+\left\langle \mathbf{x},\underset{\alpha\in\bar{E}\cup V}{\sum}\mathbf{z}_{\alpha}\right\rangle -\left\langle \bar{\mathbf{x}},\underset{\alpha\in\bar{E}\cup V}{\sum}\mathbf{z}_{\alpha}\right\rangle +\frac{1}{2}\left\Vert \underset{\alpha\in\bar{E}\cup V}{\sum}\mathbf{z}_{\alpha}\right\Vert ^{2}\end{array}\nonumber \\
 & = & \begin{array}{c}
\frac{1}{2}\left\Vert \bar{\mathbf{x}}-\mathbf{x}-\underset{\alpha\in\bar{E}\cup V}{\sum}\mathbf{z}_{\alpha}\right\Vert ^{2}\geq0.\end{array}\nonumber 
\end{eqnarray}
Since the estimate of the primal solution $\mathbf{x}^{n,w}$ is $\bar{\mathbf{x}}-\sum_{\alpha\in V\cup\bar{E}}\mathbf{z}_{\alpha}^{n,w}$
by \eqref{eq_m:all_acronyms}, substituting $\mathbf{x}$ in \eqref{eq:From-8}
to be the primal optimal solution \textbf{$\mathbf{x}^{*}$ }shows
that the difference of the dual objective value and its optimal value
bounds the distance $\frac{1}{2}\|\mathbf{x}^{n,w}-\mathbf{x}^{*}\|^{2}$.
For the rest of this paper, we shall be looking at the rate of convergence
of the dual objective value \eqref{eq:h-def-1} to its optimum value.
This remark justifies the usefulness of our results.
\end{rem}

\end{rem}

\section{\label{sec:Lin-conv}Linear convergence when all functions are smooth}

Throughout this section, we make the following assumption. 
\begin{assumption}
\label{assu:lin-conv-assu}For the problem \eqref{eq:general-framework},
we make the following assumptions: 

\begin{enumerate}
\item $V=V_{4}$. In other words, $V_{1}=V_{2}=V_{3}=\emptyset$. 
\item The sets $X_{i}=\mathbb{R}^{m}$, $\bar{E}=E\times\{1,\dots,m\}$
and $\{H_{\alpha}\}_{\alpha\in\bar{E}}$ are as described in Example
\ref{exa:distrib-dyk}.
\item For all $i\in V_{4}$, $f_{i}^{*}(\cdot)$ is strongly convex with
modulus $\sigma>0$, which is equivalent to $\nabla f_{i}(\cdot)$
being Lipschitz continuous with constant $\frac{1}{\sigma}$. {[}Note
that in general, $f_{i}(\cdot)$ are subdifferentiable for all $i\in V_{4}$,
but we now limit to only smooth $f_{i}(\cdot)$.{]} 
\end{enumerate}
\end{assumption}

For the problem \eqref{eq:distrib-dyk-primal-pblm} in Example \ref{exa:distrib-dyk}
satisfying Assumption \ref{assu:lin-conv-assu}, a primal method,
for example \cite{Gurbuzbalaban_Ozdaglar_Parrilo_SIOPT_2017,Aytekin_F_Johansson_2016},
can achieve linear convergence. Since Algorithm \ref{alg:Ext-Dyk}
has the features explained in Subsection \ref{subsec:Dist-Dyk}, it
is of interest to find out whether Algorithm \ref{alg:Ext-Dyk} also
has linear convergence under Assumption \ref{assu:lin-conv-assu}.
We shall prove such a result in this section.

We write down the function $F_{S}:\mathbf{X}^{|V\cup\bar{E}|}\to\mathbb{R}\cup\{\infty\}$
to be minimized. 
\begin{equation}
\begin{array}{c}
\underset{\mathbf{z}_{\alpha}\in\mathbf{X}:\alpha\in V\cup\bar{E}}{\min}F_{S}(\{\mathbf{z}_{\alpha}\}_{\alpha\in V\cup\bar{E}}):=\underset{i\in V}{\sum}\mathbf{f}_{i}^{*}(\mathbf{z}_{i})+\underset{\alpha\in\bar{E}}{\sum}\delta_{H_{\alpha}}^{*}(\mathbf{z}_{\alpha})+\frac{1}{2}\Big\|\bar{\mathbf{x}}-\underset{\alpha\in V\cup\bar{E}}{\sum}\mathbf{z}_{\alpha}\Big\|^{2}.\end{array}\label{eq:F-S-problem}
\end{equation}
Note that $F_{S}(\cdot)$ is related to the dual function $F(\cdot)$
in \eqref{eq:h-def} by a sign change and a constant. 

For a set of dual variables $\mathbf{z}^{0}\in\mathbf{X}^{|V\cup\bar{E}|}$,
let $\mathbf{x}^{0}$, \textbf{$\mathbf{v}_{H}^{0}$ }and $\mathbf{v}_{A}^{0}$
be related via \eqref{eq_m:all_acronyms}. Note that $\mathbf{v}_{H}^{0}=\sum_{\alpha\in\bar{E}}\mathbf{z}_{\alpha}^{0}\in\sum_{\alpha\in\bar{E}}H_{\alpha}^{\perp}\subset D^{\perp}$.
We write $\hat{z}^{0}\in\mathbb{R}^{m}$ as 
\begin{eqnarray}
\hat{z}^{0} & := & \begin{array}{c}
-\frac{1}{|V|}\underset{i\in V}{\overset{\phantom{i\in V}}{\sum}[}\mathbf{x}^{0}]_{i}\overset{\eqref{eq_m:all_acronyms}}{=}-\frac{1}{|V|}\underset{i\in V}{\overset{\phantom{i\in V}}{\sum}}\Big[\bar{\mathbf{x}}-\mathbf{v}_{H}^{0}-\underset{j\in V}{\sum}\mathbf{z}_{j}^{0}\Big]_{i}\end{array}\nonumber \\
 & \overset{\scriptsize{\mbox{Prop }\ref{prop:sparsity}}}{=} & \begin{array}{c}
-\frac{1}{|V|}\underset{i\in V}{\overset{\phantom{i\in V}}{\sum}}\Big[\bar{\mathbf{x}}-\mathbf{v}_{H}^{0}-\mathbf{z}_{i}^{0}\Big]_{i}\overset{\eqref{eq:D-formula-b}}{=}\frac{1}{|V|}\underset{i\in V}{\overset{\phantom{i\in V}}{\sum}}[\mathbf{z}_{i}^{0}-\bar{\mathbf{x}}]_{i}.\end{array}\label{eq:z-hat-zero}
\end{eqnarray}
In other words, the vector $\hat{\mathbf{z}}^{0}:=(\hat{z}^{0},\dots,\hat{z}^{0})$
is the projection of $-\mathbf{x}^{0}$ onto $D$ as defined in \eqref{eq:D-formula-a}.
Define $\mathbf{e}\in[\mathbb{R}^{m}]^{|V|}$ by $[\mathbf{e}]_{i}:=[\bar{\mathbf{x}}-\mathbf{v}_{H}^{0}-\mathbf{z}_{i}^{0}]_{i}+[\hat{\mathbf{z}}^{0}]_{i}$
so that $[\bar{\mathbf{x}}-\mathbf{v}_{H}^{0}-\mathbf{z}_{i}^{0}]_{i}=[-\hat{\mathbf{z}}^{0}+\mathbf{e}]_{i}$.
The value $F_{S}(\{\mathbf{z}_{\alpha}^{0}\}_{\alpha\in V\cup\bar{E}})$
can be written as 
\[
\begin{array}{c}
F_{S}(\{\mathbf{z}_{\alpha}^{0}\}_{\alpha\in V\cup\bar{E}})=\underset{i\in V}{\overset{\phantom{i\in V}}{\sum}}\left[f_{i}^{*}([\mathbf{z}_{i}^{0}]_{i})+\frac{1}{2}\|-\hat{z}^{0}+[\mathbf{e}]_{i}\|^{2}\right].\end{array}
\]
Let $\mathbf{z}^{*}$ be a minimizer of $F_{S}(\cdot)$. The strong
convexity of $f_{i}^{*}(\cdot)$ from Assumption \ref{assu:lin-conv-assu}(3)
ensures that \textbf{$\mathbf{z}_{i}^{*}$ }is unique if $i\in V$.
(Though $\mathbf{z}_{\alpha}^{*}$ need not be unique if $\alpha\in\bar{E}$.)
The unique solution has the value 
\[
\begin{array}{c}
F_{S}^{*}:=F_{S}(\mathbf{z}^{*})=\underset{i\in V}{\overset{\phantom{i\in V}}{\sum}}\left[f_{i}^{*}([\mathbf{z}_{i}^{*}]_{i})+\frac{1}{2}\|-\hat{z}^{*}\|^{2}\right],\end{array}
\]
where $\hat{z}^{*}$ is defined to be $\hat{z}^{*}=-\frac{1}{|V|}\sum_{i\in V}[\bar{\mathbf{x}}-\mathbf{v}_{H}^{*}-\mathbf{z}_{i}^{*}]_{i}=\frac{1}{|V|}\sum_{i\in V}[\mathbf{z}_{i}^{*}-\bar{\mathbf{x}}]_{i}$.
Note that $(-\hat{z}^{*},\dots,-\hat{z}^{*})$ is the projection of
$\mathbf{x}^{*}=\bar{\mathbf{x}}-\mathbf{v}_{H}^{*}-\sum_{i\in V}\mathbf{z}_{i}^{*}$
onto $D$.  For $\mathbf{v}_{H}^{*}$ to be optimal, we need all
the components of $\bar{\mathbf{x}}-\mathbf{v}_{H}^{*}-\sum_{i\in V}\mathbf{z}_{i}^{*}$
to have the same components so that $\frac{1}{2}\|\bar{\mathbf{x}}-\mathbf{v}_{H}^{*}-\sum_{i\in V}\mathbf{z}_{i}^{*}\|^{2}$,
which appears in the dual objective function \eqref{eq:h-def-1} through
\eqref{eq:v-H-def}, has minimum norm. This leads to $\mathbf{x}^{*}$
having all components being $-\hat{z}^{*}$.  
\begin{lem}
\label{lem:gamma-ineq}Suppose Assumption \ref{assu:lin-conv-assu}(2)
and (3) hold. Suppose $\{\mathbf{z}_{\alpha}^{0}\}_{\alpha\in V\cup\bar{E}}$
is a dual variable, and let the derived variables $\mathbf{x}^{0}$,
\textbf{$\mathbf{v}_{H}^{0}$, }$\mathbf{v}_{A}^{0}$ and $\mathbf{e}$
be as defined in the above commentary. Let $\mathbf{z}_{i}^{+}\in[\mathbb{R}^{m}]^{|V|}$
be defined so that $[\mathbf{z}_{i}^{+}]_{j}=0$ when $i\neq j$ and
\begin{equation}
\begin{array}{c}
[\mathbf{z}_{i}^{+}]_{i}=\underset{z\in\mathbb{R}^{m}}{\overset{\phantom{z\in\mathbb{R}^{m}}}{\arg\min}}f_{i}^{*}(z)+\frac{1}{2}\|-\hat{z}^{0}+[\mathbf{z}_{i}^{0}+\mathbf{e}_{i}]-z\|^{2}\mbox{ for all }i\in V.\end{array}\label{eq:z-plus-minimizer}
\end{equation}
For all $i\in V$, let $\tilde{f}_{i}:\mathbb{R}^{m}\to\mathbb{R}$
and $\tilde{\mathbf{f}}_{i}:[\mathbb{R}^{m}]^{|V|}\to\mathbb{R}$
be related through $\tilde{\mathbf{f}}_{i}(\mathbf{x})=\tilde{f}_{i}([\mathbf{x}]_{i})$.
Assume also that $\tilde{f}_{i}(\cdot)\leq f_{i}(\cdot)$, which is
equivalent to $\tilde{f}_{i}^{*}(\cdot)\geq f_{i}^{*}(\cdot)$. Let
$\tilde{F}_{S}:[[\mathbb{R}^{m}]^{|V|}]^{|V\cup\bar{E}|}\to\mathbb{R}\cup\{\infty\}$
be defined in a manner similar to \eqref{eq:F-S-problem} as 
\begin{equation}
\begin{array}{c}
\tilde{F}_{S}(\{\mathbf{z}_{\alpha}\}_{\alpha\in V\cup\bar{E}}):=\underset{i\in V}{\sum}\tilde{\mathbf{f}}_{i}^{*}(\mathbf{z}_{i})+\underset{\alpha\in\bar{E}}{\sum}\delta_{H_{\alpha}}^{*}(\mathbf{z}_{\alpha})+\frac{1}{2}\Big\|\bar{\mathbf{x}}-\underset{\alpha\in V\cup\bar{E}}{\sum}\mathbf{z}_{\alpha}\Big\|^{2}.\end{array}\label{eq:F-S-tilde}
\end{equation}
Then one can check that 
\begin{eqnarray*}
 &  & \begin{array}{c}
F_{S}(\{\mathbf{z}_{i}\}_{i\in V},\{\mathbf{z}_{\alpha}^{0}\}_{\alpha\in\bar{E}})=\underset{i\in V}{\sum}\left[f_{i}^{*}([\mathbf{z}_{i}]_{i})+\frac{1}{2}\|-\hat{z}^{0}+[\mathbf{z}_{i}^{0}+\mathbf{e}-\mathbf{z}_{i}]_{i}\|^{2}\right]\end{array}\\
 & \mbox{ and} & \begin{array}{c}
\tilde{F}_{S}(\{\mathbf{z}_{i}\}_{i\in V},\{\mathbf{z}_{\alpha}^{0}\}_{\alpha\in\bar{E}})=\underset{i\in V}{\overset{\phantom{i\in V}}{\sum}}\left[\tilde{f}_{i}^{*}([\mathbf{z}_{i}]_{i})+\frac{1}{2}\|-\hat{z}^{0}+[\mathbf{z}_{i}^{0}+\mathbf{e}-\mathbf{z}_{i}]_{i}\|^{2}\right].\end{array}
\end{eqnarray*}
Let $\mathbf{z}^{*}$ be a minimizer of $F_{S}(\cdot)$, and let $\mathbf{z}_{e}^{+}=\mathbf{z}_{e}^{0}$
for all $e\in\bar{E}$. Then there exists constants $\gamma\in(0,1)$
and $M>0$ such that if $\mathbf{z}^{0}$ and $\mathbf{z}^{+}$ are
related as described, then 
\begin{eqnarray}
\begin{array}{c}
F_{S}(\mathbf{z}^{+})-F_{S}(\mathbf{z}^{*})\end{array} & \leq & \begin{array}{c}
\gamma[F_{S}(\mathbf{z}^{0})-F_{S}(\mathbf{z}^{*})]+M\underset{i\in V}{\overset{\phantom{i\in V}}{\sum}}\|[\mathbf{e}]_{i}\|^{2}\end{array}\label{eq:lemma-lin-conv-last}\\
 & \leq & \begin{array}{c}
\gamma[\tilde{F}_{S}(\mathbf{z}^{0})-F_{S}(\mathbf{z}^{*})]+M\underset{i\in V}{\overset{\phantom{i\in V}}{\sum}}\|[\mathbf{e}]_{i}\|^{2}.\end{array}\nonumber 
\end{eqnarray}
\end{lem}

\begin{proof}
The second inequality of \eqref{eq:lemma-lin-conv-last} is obvious
from $\tilde{f}_{i}^{*}(\cdot)\geq f_{i}^{*}(\cdot)$. We prove the
first inequality. By \eqref{eq:z-plus-minimizer} and Assumption \ref{assu:lin-conv-assu}(3),
we have, for all $i\in V$, 
\begin{eqnarray}
 &  & \begin{array}{c}
f_{i}^{*}([\mathbf{z}_{i}^{0}]_{i})+\frac{1}{2}\|-\hat{z}^{0}+[\mathbf{e}]_{i}\|^{2}\end{array}\label{eq:f-i-original}\\
 & \overset{\eqref{eq:z-plus-minimizer}}{\geq} & \begin{array}{c}
f_{i}^{*}([\mathbf{z}_{i}^{+}]_{i})+\frac{1}{2}\|-\hat{z}^{0}+[\mathbf{z}_{i}^{0}-\mathbf{z}_{i}^{+}+\mathbf{e}]_{i}\|^{2}+\frac{1+\sigma}{2}\|[\mathbf{z}_{i}^{0}-\mathbf{z}_{i}^{+}]\|^{2}.\end{array}\nonumber 
\end{eqnarray}
Also, the optimality condition of \eqref{eq:z-plus-minimizer} implies
that $-\hat{z}^{0}+[\mathbf{z}_{i}^{0}-\mathbf{z}_{i}^{+}+\mathbf{e}]_{i}\in\partial f_{i}([\mathbf{z}_{i}^{+}]_{i})$,
so together with Assumption \ref{assu:lin-conv-assu}(3), we have
\begin{eqnarray}
 &  & \begin{array}{c}
f_{i}^{*}([\mathbf{z}_{i}^{*}]_{i})+\frac{1}{2}\|-\hat{z}^{*}\|^{2}\end{array}\nonumber \\
 & \geq & \begin{array}{c}
f_{i}^{*}([\mathbf{z}_{i}^{+}]_{i})+\frac{1}{2}\|-\hat{z}^{*}\|^{2}\end{array}\nonumber \\
 &  & \begin{array}{c}
+\langle-\hat{z}^{0}+[\mathbf{z}_{i}^{0}-\mathbf{z}_{i}^{+}+\mathbf{e}]_{i},[\mathbf{z}_{i}^{*}-\mathbf{z}_{i}^{+}]_{i}\rangle+\frac{\sigma}{2}\|[\mathbf{z}_{i}^{*}-\mathbf{z}_{i}^{+}]_{i}\|^{2}\end{array}\nonumber \\
 & = & \begin{array}{c}
f_{i}^{*}([\mathbf{z}_{i}^{+}]_{i})+\frac{1}{2}\|-\hat{z}^{0}+[\mathbf{z}_{i}^{0}-\mathbf{z}_{i}^{+}+\mathbf{e}]_{i}\|^{2}+\frac{\sigma}{2}\|[\mathbf{z}_{i}^{*}-\mathbf{z}_{i}^{+}]_{i}\|^{2}\end{array}\label{eq:f-i-to-opt}\\
 &  & \begin{array}{c}
+\frac{1}{2}\|-\hat{z}^{*}\|^{2}+\langle-\hat{z}^{0}+[\mathbf{z}_{i}^{0}-\mathbf{z}_{i}^{+}+\mathbf{e}]_{i},[\mathbf{z}_{i}^{*}-\mathbf{z}_{i}^{+}]_{i}\rangle\end{array}\nonumber \\
 &  & \begin{array}{c}
-(\frac{1}{2}\|\hat{z}^{0}\|^{2}+\frac{1}{2}\|[\mathbf{z}_{i}^{0}-\mathbf{z}_{i}^{+}]_{i}\|^{2}+\frac{1}{2}\|[\mathbf{e}]_{i}\|^{2}\end{array}\nonumber \\
 &  & \begin{array}{c}
\qquad+\langle-\hat{z}^{0},[\mathbf{z}_{i}^{0}-\mathbf{z}_{i}^{+}]_{i}\rangle+\langle[\mathbf{e}]_{i},-\hat{z}^{0}\rangle+\langle[\mathbf{e}]_{i},[\mathbf{z}_{i}^{0}-\mathbf{z}_{i}^{+}]_{i}\rangle).\end{array}\nonumber 
\end{eqnarray}
For the terms not involving $[\mathbf{e}]_{i}$ in the last formula
of \eqref{eq:f-i-to-opt}, we have 
\begin{eqnarray}
\begin{array}{c}
\underset{i\in V}{\overset{\phantom{i\in V}}{\sum}}\langle[\mathbf{z}_{i}^{0}-\mathbf{z}_{i}^{+}]_{i},[\mathbf{z}_{i}^{*}-\mathbf{z}_{i}^{+}]_{i}\rangle\geq\underset{i\in V}{\overset{\phantom{i\in V}}{\sum}}\left[-\frac{1}{2\epsilon}\|[\mathbf{z}_{i}^{0}-\mathbf{z}_{i}^{+}]_{i}\|^{2}-\frac{\epsilon}{2}\|[\mathbf{z}_{i}^{*}-\mathbf{z}_{i}^{+}]_{i}\|^{2}\right],\end{array}\label{eq:term-1}
\end{eqnarray}
and 
\begin{eqnarray}
 &  & \begin{array}{c}
\underset{i\in V}{\sum}[\langle-\hat{z}^{0},[\mathbf{z}_{i}^{*}-\mathbf{z}_{i}^{+}]_{i}\rangle+\langle\hat{z}^{0},[\mathbf{z}_{i}^{0}-\mathbf{z}_{i}^{+}]_{i}\rangle+\frac{1}{2}\|-\hat{z}^{*}\|^{2}-\frac{1}{2}\|\hat{z}^{0}\|^{2}]\end{array}\nonumber \\
 & = & \begin{array}{c}
\underset{i\in V}{\overset{\phantom{i\in V}}{\sum}}[\langle\hat{z}^{0},[\mathbf{z}_{i}^{0}-\mathbf{z}_{i}^{*}]_{i}\rangle+\frac{1}{2}\|-\hat{z}^{*}\|^{2}-\frac{1}{2}\|\hat{z}^{0}\|^{2}]\end{array}\label{eq:block-1}\\
 & = & \begin{array}{c}
|V|\langle\hat{z}^{0},\hat{z}^{0}-\hat{z}^{*}\rangle+\frac{|V|}{2}\|\hat{z}^{*}\|^{2}-\frac{|V|}{2}\|\hat{z}^{0}\|^{2}=\frac{|V|}{2}\|\hat{z}^{0}-\hat{z}^{*}\|^{2}\geq0.\end{array}\nonumber 
\end{eqnarray}
For the terms involving $[\mathbf{e}]_{i}$ in the last formula in
\eqref{eq:f-i-to-opt}, we have
\begin{eqnarray}
\begin{array}{c}
\underset{i\in V}{\overset{\phantom{i\in V}}{\sum}}\langle[\mathbf{e}]_{i},[\mathbf{z}_{i}^{*}-\mathbf{z}_{i}^{+}]_{i}\rangle\end{array} & \geq & \begin{array}{c}
\underset{i\in V}{\overset{\phantom{i\in V}}{\sum}}\left[-\frac{1}{2\epsilon}\|[\mathbf{e}]_{i}\|^{2}-\frac{\epsilon}{2}\|[\mathbf{z}_{i}^{*}-\mathbf{z}_{i}^{+}]_{i}\|^{2}\right],\end{array}\nonumber \\
\begin{array}{c}
\underset{i\in V}{\overset{\phantom{i\in V}}{\sum}}\langle[\mathbf{e}]_{i},-\hat{z}^{0}\rangle\end{array} & = & \begin{array}{c}
\left\langle \underset{i\in V}{\overset{\phantom{i\in V}}{\sum}}[\mathbf{e}]_{i},-\hat{z}^{0}\right\rangle =0,\end{array}\label{eq:block-2}\\
\begin{array}{c}
\underset{i\in V}{\overset{\phantom{i\in V}}{\sum}}\langle[\mathbf{e}]_{i},[\mathbf{z}_{i}^{0}-\mathbf{z}_{i}^{+}]_{i}\rangle\end{array} & \geq & \begin{array}{c}
\underset{i\in V}{\overset{\phantom{i\in V}}{\sum}}\left[-\frac{\epsilon}{2}\|[\mathbf{e}]_{i}\|^{2}-\frac{1}{2\epsilon}\|[\mathbf{z}_{i}^{0}-\mathbf{z}_{i}^{+}]_{i}\|^{2}\right].\end{array}\nonumber 
\end{eqnarray}
Summing up the right hand sides of \eqref{eq:term-1}, \eqref{eq:block-1}
and \eqref{eq:block-2} and $\sum_{i\in V}[\frac{\sigma}{2}\|[\mathbf{z}_{i}^{*}-\mathbf{z}_{i}^{+}]_{i}\|^{2}-\frac{1}{2}\|[\mathbf{e}]_{i}\|^{2}-\frac{1}{2}\|[\mathbf{z}_{i}^{0}-\mathbf{z}_{i}^{+}]_{i}\|^{2}]$
and setting $\epsilon=\sigma/2$ gives 
\begin{eqnarray}
 &  & \begin{array}{c}
\underset{i\in V}{\sum}\big[(\frac{\sigma}{2}-\epsilon)\|[\mathbf{z}_{i}^{*}-\mathbf{z}_{i}^{+}]_{i}\|^{2}-(\frac{1}{\epsilon}+\frac{1}{2})\|[\mathbf{z}_{i}^{0}-\mathbf{z}_{i}^{+}]_{i}\|^{2}-[\frac{1}{2\epsilon}+\frac{1}{2}+\frac{\epsilon}{2}]\|[\mathbf{e}]_{i}\|^{2}\big]\end{array}\nonumber \\
 & = & \begin{array}{c}
\underset{i\in V}{\overset{\phantom{i\in V}}{\sum}}\big[-(\frac{2}{\sigma}+\frac{1}{2})\|[\mathbf{z}_{i}^{0}-\mathbf{z}_{i}^{+}]_{i}\|^{2}-[\frac{1}{\sigma}+\frac{1}{2}+\frac{\sigma}{4}]\|[\mathbf{e}]_{i}\|^{2}\big].\end{array}\label{eq:block-3}
\end{eqnarray}
Summing the formulas in \eqref{eq:f-i-to-opt} to \eqref{eq:block-3},
we have 
\begin{eqnarray}
 &  & \begin{array}{c}
\underset{i\in V}{\sum}\left[f_{i}^{*}([\mathbf{z}_{i}^{*}]_{i})+\frac{1}{2}\|-\hat{z}^{*}\|^{2}\right]\end{array}\nonumber \\
 & \overset{\scriptsize{\text{\eqref{eq:f-i-to-opt} to \eqref{eq:block-3}}}}{\geq} & \begin{array}{c}
\underset{i\in V}{\overset{\phantom{i\in V}}{\sum}}\left[f_{i}^{*}([\mathbf{z}_{i}^{+}]_{i})+\frac{1}{2}\|-\hat{z}^{0}+[\mathbf{z}_{i}^{0}-\mathbf{z}_{i}^{+}+\mathbf{e}]_{i}\|^{2}\right]\end{array}\label{eq:f-i-to-opt-better}\\
 &  & \begin{array}{c}
+\underset{i\in V}{\sum}\left[-(\frac{2}{\sigma}+\frac{1}{2})\|[\mathbf{z}_{i}^{0}-\mathbf{z}_{i}^{+}]_{i}\|^{2}-[\frac{1}{\sigma}+\frac{1}{2}+\frac{\sigma}{4}]\|[\mathbf{e}]_{i}\|^{2}\right].\end{array}\nonumber 
\end{eqnarray}
We can then sum up \eqref{eq:f-i-original} multiplied by $\frac{4+\sigma}{\sigma(1+\sigma)}$
and \eqref{eq:f-i-to-opt-better} to get 
\begin{eqnarray}
 &  & \begin{array}{c}
\left(\frac{4+\sigma}{\sigma(1+\sigma)}\right)\underset{i\in V}{\overset{\phantom{i\in V}}{\sum}}[f_{i}^{*}([\mathbf{z}_{i}^{0}]_{i})+\frac{1}{2}\|-\hat{z}^{0}+[\mathbf{e}]_{i}\|^{2}]+\underset{i\in V}{\sum}\left[f_{i}^{*}([\mathbf{z}_{i}^{*}]_{i})+\frac{1}{2}\|-\hat{z}^{*}\|^{2}\right]\end{array}\nonumber \\
 & \geq & \begin{array}{c}
\left(\frac{4+\sigma}{\sigma(1+\sigma)}+1\right)\underset{i\in V}{\overset{\phantom{i\in V}}{\sum}}\left[f_{i}^{*}([\mathbf{z}_{i}^{+}]_{i})+\frac{1}{2}\|-\hat{z}^{0}+[\mathbf{z}_{i}^{0}-\mathbf{z}_{i}^{+}+\mathbf{e}]_{i}\|^{2}\right]\end{array}\label{eq:F-S-lin-origin}\\
 &  & \begin{array}{c}
-[\frac{1}{\sigma}+\frac{1}{2}+\frac{\sigma}{4}]\underset{i\in V}{\sum}\|[\mathbf{e}]_{i}\|^{2}.\end{array}\nonumber 
\end{eqnarray}
Letting $\gamma=\left(\frac{4+\sigma}{\sigma(1+\sigma)}\right)/\left(\frac{4+\sigma}{\sigma(1+\sigma)}+1\right)$,
we can rearrange \eqref{eq:F-S-lin-origin} to get\textbf{ }the first
inequality in \eqref{eq:lemma-lin-conv-last} with $M=(1-\gamma)[\frac{1}{\sigma}+\frac{1}{2}+\frac{\sigma}{4}]$.
This concludes the proof. 
\end{proof}
We now proceed to prove the linear convergence result. Let $\tilde{F}_{S}^{n,w}(\cdot)$
be defined by 
\begin{equation}
\begin{array}{c}
\tilde{F}_{S}^{n,w}(\{\mathbf{z}_{\alpha}\}_{\alpha\in V\cup\bar{E}}):=\underset{i\in V}{\sum}\mathbf{f}_{i,n,w}^{*}(\mathbf{z}_{i})+\underset{\alpha\in\bar{E}}{\sum}\delta_{H_{\alpha}}^{*}(\mathbf{z}_{\alpha})+\frac{1}{2}\Big\|\bar{\mathbf{x}}-\underset{\alpha\in V\cup\bar{E}}{\overset{\phantom{\alpha\in V\cup\bar{E}}}{\sum}}\mathbf{z}_{\alpha}\Big\|^{2}.\end{array}\label{eq:F-S-tilde-accent}
\end{equation}

\begin{thm}
(Linear convergence) Suppose Assumption \ref{assu:lin-conv-assu}
holds. Consider Algorithm \ref{alg:Ext-Dyk} being applied to solve
\eqref{eq:general-framework}. Suppose $S_{n,1}=V$ and $\mathbf{z}^{*}$
is a minimizer of $F_{S}(\cdot)$. Suppose also that $\cup_{w=1}^{\bar{w}}S_{n,w}\cap\bar{E}=\bar{E}_{n}$.
Then there is some $c\in(0,1)$ such that 
\[
\tilde{F}_{S}^{n+1,0}(\mathbf{z}^{n+1,0})-F_{S}(\mathbf{z}^{*})\leq c[\tilde{F}_{S}^{n,0}(\mathbf{z}^{n,0})-F_{S}(\mathbf{z}^{*})].
\]
\end{thm}

\begin{proof}
Let $\mathbf{z}^{+}\in\mathbf{X}^{|V\cup\bar{E}|}$ satisfy \eqref{eq:z-plus-minimizer},
where $\mathbf{z}^{0}=\mathbf{z}^{n,0}$ and $\hat{z}^{0}\overset{\eqref{eq:z-hat-zero}}{=}\frac{1}{|V|}\sum_{i\in V}[\mathbf{z}_{i}^{0}-\bar{\mathbf{x}}]_{i}$
like in Lemma \ref{lem:gamma-ineq}. Lemma \ref{lem:gamma-ineq} shows
that there is a $\gamma\in(0,1)$ such that 
\[
\begin{array}{c}
F_{S}(\mathbf{z}^{+})-F_{S}(\mathbf{z}^{*})\overset{\eqref{eq:lemma-lin-conv-last}}{\leq}\gamma[\tilde{F}_{S}^{n,0}(\mathbf{z}^{n,0})-F_{S}(\mathbf{z}^{*})]+M\underset{i\in V}{\sum}\|\mathbf{e}_{i}\|^{2}.\end{array}
\]
Let $p_{i}=-\hat{z}^{0}+[\mathbf{z}_{i}^{0}+\mathbf{e}]_{i}$, which
is the proximal center in the formula \eqref{eq:z-plus-minimizer}.
We make use of Lemma \ref{lem:subdiff-decrease}(2) to see that there
is a constant $\gamma_{2}\in[0,1)$ such that, for all $i\in V$,
\begin{eqnarray}
 &  & \begin{array}{c}
[f_{i,n,1}^{*}([\mathbf{z}_{i}^{n,1}]_{i})+\frac{1}{2}\|p_{i}-[\mathbf{z}_{i}^{n,1}]_{i}\|^{2}]-[f_{i}^{*}([\mathbf{z}_{i}^{+}]_{i})+\frac{1}{2}\|p_{i}-[\mathbf{z}_{i}^{+}]_{i}\|^{2}]\end{array}\label{eq:lin-conv-component}\\
 & \leq & \begin{array}{c}
\gamma_{2}\big[[f_{i,n,0}^{*}([\mathbf{z}_{i}^{n,0}]_{i})+\frac{1}{2}\|p_{i}-[\mathbf{z}_{i}^{n,0}]_{i}\|^{2}]-[f_{i}^{*}([\mathbf{z}_{i}^{+}]_{i})+\frac{1}{2}\|p_{i}-[\mathbf{z}_{i}^{+}]_{i}\|^{2}]\big].\end{array}\nonumber 
\end{eqnarray}
Summing \eqref{eq:lin-conv-component} over all $i\in V$ applied
to \eqref{eq:F-S-tilde-accent} gives 
\begin{equation}
\tilde{F}_{S}^{n,1}(\mathbf{z}^{n,1})-F_{S}(\mathbf{z}^{+})\leq\gamma_{2}[\tilde{F}_{S}^{n,0}(\mathbf{z}^{n,0})-F_{S}(\mathbf{z}^{+})].\label{eq:gamma-2-ineq}
\end{equation}
Then we have 
\begin{eqnarray}
 &  & \tilde{F}_{S}^{n,1}(\mathbf{z}^{n,1})-F_{S}(\mathbf{z}^{*})\label{eq:chunk-ineq}\\
 & = & \tilde{F}_{S}^{n,1}(\mathbf{z}^{n,1})-F_{S}(\mathbf{z}^{+})+F_{S}(\mathbf{z}^{+})-F_{S}(\mathbf{z}^{*})\nonumber \\
 & \overset{\eqref{eq:gamma-2-ineq}}{\leq} & \gamma_{2}[\tilde{F}_{S}^{n,0}(\mathbf{z}^{n,0})-F_{S}(\mathbf{z}^{+})]+F_{S}(\mathbf{z}^{+})-F_{S}(\mathbf{z}^{*})\nonumber \\
 & = & \gamma_{2}[\tilde{F}_{S}^{n,0}(\mathbf{z}^{n,0})-F_{S}(\mathbf{z}^{*})]+(1-\gamma_{2})[F_{S}(\mathbf{z}^{+})-F_{S}(\mathbf{z}^{*})]\nonumber \\
 & \overset{\scriptsize{\mbox{Lem }\ref{lem:gamma-ineq}}}{\leq} & \underbrace{[1-(1-\gamma)(1-\gamma_{2})]}_{\gamma_{3}}[\tilde{F}_{S}^{n,0}(\mathbf{z}^{n,0})-F_{S}(\mathbf{z}^{*})]+(1-\gamma_{2})M\sum_{i\in V}\|\mathbf{e}_{i}\|^{2}.\nonumber 
\end{eqnarray}
Since $\gamma<1$ and $\gamma_{2}<1$, the $\gamma_{3}$ as marked
above satisfies $\gamma_{3}\in[0,1)$. We now consider 2 cases. 

\textbf{Case 1: $(1-\gamma_{2})M\sum_{i\in V}\|\mathbf{e}_{i}\|^{2}\leq\frac{1-\gamma_{3}}{2}[\tilde{F}_{S}^{n,0}(\mathbf{z}^{n,0})-F_{S}(\mathbf{z}^{*})]$. }

In this case, we make use of Theorem \ref{thm:convergence}(i) to
get 
\[
\begin{array}{c}
\tilde{F}_{S}^{n+1,0}(\mathbf{z}^{n+1,0})-F_{S}(\mathbf{z}^{*})\leq\tilde{F}_{S}^{n,1}(\mathbf{z}^{n,1})-F_{S}(\mathbf{z}^{*})\overset{\eqref{eq:chunk-ineq}}{\leq}\frac{1+\gamma_{3}}{2}[\tilde{F}_{S}^{n,0}(\mathbf{z}^{n,0})-F_{S}(\mathbf{z}^{*})].\end{array}
\]

\textbf{Case 2: $(1-\gamma_{2})M\sum_{i\in V}\|\mathbf{e}_{i}\|^{2}\geq\frac{1-\gamma_{3}}{2}[\tilde{F}_{S}^{n,0}(\mathbf{z}^{n,0})-F_{S}(\mathbf{z}^{*})]$. }

We recall that $\mathbf{x}^{n,w}\overset{\eqref{eq_m:all_acronyms}}{=}\bar{\mathbf{x}}-\sum_{\alpha\in V\cup\bar{E}}\mathbf{z}_{\alpha}^{n,w}$.
The value $\sum_{i\in V}\|[\mathbf{e}]_{i}\|^{2}$ equals $d(\mathbf{x}^{n,0},D)^{2}$,
where $D$ is the diagonal set \eqref{eq:D-formula-a}. By the Hoffman
lemma and the fact that $\bar{E}_{n}$ connects $V$, there is a constant
$\kappa$ such that $\kappa d(\mathbf{x},D)\leq\max_{\alpha\in\bar{E}_{n}}d(\mathbf{x},H_{\alpha})$.
Let $\alpha^{*}\in\bar{E}_{n}$ be such that $\kappa d(\mathbf{x}^{n,0},D)\leq d(\mathbf{x}^{n,0},H_{\alpha^{*}})$.
Let $w_{\alpha^{*}}$ be in $\{1,\dots,\bar{w}\}$ such that $\alpha^{*}\in S_{n,w_{\alpha^{*}}}$.
We have $\mathbf{x}^{n,w_{\alpha^{*}}}\in H_{\alpha^{*}}$. (To see
this, note that the dual of \eqref{eq:Dykstra-min-subpblm}, is a
primal problem involving $\delta_{H_{\alpha^{*}}}(\cdot)$.) To summarize,
\begin{eqnarray}
\begin{array}{c}
\frac{1-\gamma_{3}}{2M(1-\gamma_{2})}[\tilde{F}_{S}^{n,0}(\mathbf{z}^{n,0})-F_{S}(\mathbf{z}^{*})]\overset{\scriptsize{\mbox{Case 2}}}{\leq}\underset{i\in V}{\overset{\phantom{i\in V}}{\sum}}\|[\mathbf{e}]_{i}\|^{2}\end{array} & = & \begin{array}{c}
d(\mathbf{x}^{n,0},D)^{2}\end{array}\label{eq:lin-reg-arg}\\
\begin{array}{c}
\leq\frac{1}{\kappa^{2}}d(\mathbf{x}^{n,0},H_{\alpha^{*}})^{2}\end{array} & \leq & \begin{array}{c}
\frac{1}{\kappa^{2}}\|\mathbf{x}^{n,0}-\mathbf{x}^{n,w_{\alpha^{*}}}\|^{2}.\end{array}\nonumber 
\end{eqnarray}
 Another ingredient of our proof is to make use of Theorem \ref{thm:convergence}(i)
to get  
\begin{equation}
\begin{array}{c}
\tilde{F}_{S}^{n+1,0}(\mathbf{z}^{n+1,0})\leq\tilde{F}_{S}^{n,w_{\alpha^{*}}}(\mathbf{z}^{n,w_{\alpha^{*}}})\overset{\scriptsize{\mbox{Thm \ref{thm:convergence}(i)}}}{\leq}\tilde{F}_{S}^{n,0}(\mathbf{z}^{n,0})-\frac{1}{2}\underset{w=0}{\overset{w_{\alpha^{*}}-1}{\sum}}\|\mathbf{x}^{n,w}-\mathbf{x}^{n,w+1}\|^{2}.\end{array}\label{eq:linreg-dec}
\end{equation}
Next, 
\begin{eqnarray}
 &  & \begin{array}{c}
\frac{1}{2}\underset{w=0}{\overset{w_{\alpha^{*}}-1}{\sum}}\|\mathbf{x}^{n,w}-\mathbf{x}^{n,w+1}\|^{2}\geq\frac{1}{2w_{\alpha^{*}}}\left\Vert \underset{w=0}{\overset{w_{\alpha^{*}}-1}{\sum}}[\mathbf{x}^{n,w}-\mathbf{x}^{n,w+1}]\right\Vert ^{2}\end{array}\nonumber \\
 & \geq & \begin{array}{c}
\frac{1}{2\bar{w}}\|\mathbf{x}^{n,0}-\mathbf{x}^{n,w_{\alpha^{*}}}\|^{2}\end{array}\overset{\eqref{eq:lin-reg-arg}}{\geq}\begin{array}{c}
\frac{\kappa^{2}(1-\gamma_{3})}{4M\bar{w}(1-\gamma_{2})}[\tilde{F}_{S}^{n,0}(\mathbf{z}^{n,0})-F_{S}(\mathbf{z}^{*})].\end{array}\label{eq:linreg-squares}
\end{eqnarray}
So 
\[
\begin{array}{c}
\tilde{F}_{S}^{n+1,0}(\mathbf{z}^{n+1,0})-F_{S}(\mathbf{z}^{*})\overset{\eqref{eq:linreg-dec},\eqref{eq:linreg-squares}}{\leq}\left(1-\frac{\kappa^{2}(1-\gamma_{3})}{4M\bar{w}(1-\gamma_{2})}\right)[\tilde{F}_{S}^{n,0}(\mathbf{z}^{n,0})-F_{S}(\mathbf{z}^{*})].\end{array}
\]
This completes the proof of linear convergence. 
\end{proof}
\begin{rem}
The linear convergence rates in \eqref{eq:lin-reg-arg} can plausibly
be refined with the study of effective resistances in \cite{Aybat_Gurb_effective_resistance_2017}.
\end{rem}

\section{\label{sec:O-1-k-conv-of-smooth-proximable-case}$O(1/k)$ convergence
when all functions are either smooth or proximable}

In this section, we prove the $O(1/k)$ convergence rate of the dual
objective value when all $f_{i}(\cdot)$ are smooth for all $i\in V_{4}$.
We begin by discussing conditions ensuring the boundedness of $\{\mathbf{z}^{n,w}\}$
before our main result. 

\subsection{On the boundedness of $\{\mathbf{z}^{n,w}\}$}

In this subsection, we discuss a standard constraint qualification
that ensures the boundedness of $\{\mathbf{z}^{n,w}\}$. 

We write down a lemma on functions whose domain is not the entire
space. 
\begin{lem}
\label{lem:domain-normals}Suppose $X$ is a finite dimensional Hilbert
space, and $f:X\to\mathbb{R}$ is a closed convex function. Suppose
$\{x_{i}\}_{i=1}^{\infty}$ and $\{z_{i}\}_{i=1}^{\infty}$ are sequences
such that $z_{i}\in\partial f(x_{i})$ for all $i\geq1$. Suppose
that $\lim_{i\to\infty}x_{i}$ exists, say $x^{*}$, and $\lim_{i\to\infty}\|z_{i}\|=\infty$.
Then a cluster point of $\{\frac{z_{i}}{\|z_{i}\|}\}$ lies in $N_{\scriptsize{\cl(\dom(f))}}(x^{*})$.
\end{lem}

\begin{proof}
We make use of the fact that $z\in\partial f(x)$ if and only if $(z,-1)\in N_{\scriptsize{\epi(f)}}(x,f(x))$.
By choosing subsequences, we can assume that $\lim_{i\to\infty}f(x_{i})$
exists as either a finite number or $\infty$. We can assume that
$\lim_{i\to\infty}\frac{z_{i}}{\|z_{i}\|}$ exists and equals $z^{*}$. 

Suppose $\lim_{i\to\infty}f(x_{i})$ is a finite number, say $f^{*}$.
In this case, $(x^{*},f^{*})$ lies in the epigraph of $f(\cdot)$.
By the closedness of the normal cones, $\lim_{i\to\infty}\frac{(z_{i},-1)}{\|z_{i}\|}=(z^{*},0)$
is a normal vector of $\epi(f)$ at $(x^{*},f^{*})$. This implies
that $z^{*}$ lies in $N_{\scriptsize{\cl(\dom(f))}}(x^{*})$. 

Suppose $\lim_{i\to\infty}f(x_{i})$ is infinity. Take any $x'\in\dom(f)$,
Let $I$ be large enough so that $f(x_{i})>f(x')$ for all $i>I$.
Then the point $(x',f(x_{i}))$ would lie in $\epi(f)$. Then 
\[
\begin{array}{c}
\left\langle \frac{(z_{i},-1)}{\|z_{i}\|},\big[\big(x_{i},f(x_{i})\big)-\big(x',f(x_{i})\big)\big]\right\rangle \leq0\mbox{ for all }i>I,\end{array}
\]
which gives $\langle\frac{z_{i}}{\|z_{i}\|},x_{i}-x'\rangle\leq0$
for all $i>I$. Taking limits gives $\langle z^{*},x^{*}-x'\rangle\leq0$,
which implies that $z^{*}\in N_{\scriptsize{\cl(\dom(f))}}(x^{*})$. 
\end{proof}
With the above lemma, we can prove the following result on the boundedness
of the iterates $\{\mathbf{z}_{\alpha}\}_{\alpha\in V_{1}\cup V_{2}}$
under a constraint qualification. 
\begin{thm}
\label{thm:CQ-gives-bdd-iters}Suppose we are under the setting of
Example \ref{exa:distrib-dyk}, and let $x^{*}$ be the primal minimizer
to \eqref{eq:distrib-dyk-primal-pblm}. Suppose that the condition
\begin{eqnarray}
 &  & \begin{array}{c}
y_{i}\in N_{\scriptsize{\cl(\dom(f_{i}))}}(x^{*})\mbox{ for all }i\in V_{1}\cup V_{2}\mbox{ and }\underset{i\in V_{1}\cup V_{2}}{\sum}y_{i}=0\end{array}\nonumber \\
 & \mbox{ implies } & \begin{array}{c}
y_{i}=0\mbox{ for all }i\in V_{1}\cup V_{2}.\end{array}\label{eq:CQ}
\end{eqnarray}
 holds. Then the iterates $\{\mathbf{z}_{i}^{n,w}\}_{i\in V}$ of
Algorithm \ref{alg:Ext-Dyk} are bounded.
\end{thm}

\begin{proof}
Seeking a contradiction, suppose $\{\mathbf{z}_{i}^{n,w}\}_{n,w}$
is not bounded for some $i\in V$. This means that we can find a subsequence
$\{n_{k}\}_{k}$ such that 
\begin{equation}
\begin{array}{c}
\underset{k\to\infty}{\overset{\phantom{k\to\infty}}{\lim}}[\underset{j\in V}{\max}\|\mathbf{z}_{j}^{n_{k},\bar{w}}\|]=\infty.\end{array}\label{eq:norms-goto-inf}
\end{equation}
Theorem \ref{thm:convergence}(ii), formula \eqref{eq_m:all_acronyms}
and the finite dimensionality of $\mathbf{X}$ ensures that $\{\mathbf{x}^{n_{k},\bar{w}}\}_{n_{k}=1}^{\infty}$
has a cluster point. We can choose a further subsequence if necessary
so that $\lim_{k\to\infty}\mathbf{x}^{n_{k},\bar{w}}=\hat{\mathbf{x}}$
for some $\hat{\mathbf{x}}\in\mathbf{X}$. Theorem \ref{thm:convergence}(i)
implies that $\lim_{k\to\infty}\|\mathbf{x}^{n_{k},p(n_{k},\alpha)}-\mathbf{x}^{n_{k},\bar{w}}\|=0$
for all $\alpha\in V\cup\bar{E}$, so $\lim_{k\to\infty}\mathbf{x}^{n_{k},p(n_{k},\alpha)}=\hat{\mathbf{x}}$.
Since $\mathbf{x}^{n_{k},p(n_{k},\alpha)}\in H_{\alpha}$ for all
$\alpha\in\bar{E}$, going back to the primal problem associated with
\eqref{eq:Dykstra-min-subpblm} gives us $\hat{\mathbf{x}}\in\cap_{\alpha\in\bar{E}}H_{\alpha}$.
Let $\hat{x}\in\mathbb{R}^{m}$ be any component of $\hat{\mathbf{x}}\in[\mathbb{R}^{m}]^{|V|}$.
(Note that all components in $\hat{\mathbf{x}}\in[\mathbb{R}^{m}]^{|V|}$
are equal since $\hat{\mathbf{x}}\in\cap_{\alpha\in\bar{E}}H_{\alpha}$.) 

Assume, by taking further subsequences if necessary, that for all
$i\in V$, 
\begin{equation}
\begin{array}{c}
\hat{\mathbf{z}}_{i}:=\underset{k\to\infty}{\lim}\frac{\mathbf{z}_{i}^{n_{k},\bar{w}}}{\max_{j\in V}\|\mathbf{z}_{j}^{n_{k},\bar{w}}\|}\end{array}\label{eq:z-i-hat}
\end{equation}
exists as a limit. From Lemma \ref{lem:domain-normals}, we have $[\hat{\mathbf{z}}_{i}]_{i}\in N_{\scriptsize{\cl(\dom(f_{i}))}}(\hat{x})$
for all $i\in V$. Recall that $\{\mathbf{v}_{A}^{n,\bar{w}}\}$ is
a bounded set by Theorem \ref{thm:convergence}(ii), so by \eqref{eq_m:all_acronyms},
$\{\sum_{\alpha\in V\cup\bar{E}}\mathbf{z}_{\alpha}^{n,\bar{w}}\}_{n=1}^{\infty}$
is bounded. Note that for all $\alpha\in\bar{E}$, $\mathbf{z}_{\alpha}^{n,w}\in H_{\alpha}^{\perp}\subset D^{\perp}$.
We then have 
\[
\begin{array}{c}
\underset{j\in V}{\sum}\Big[\underset{\alpha\in V\cup\bar{E}}{\sum}\mathbf{z}_{\alpha}^{n,\bar{w}}\Big]_{j}\overset{\eqref{eq:D-formula-b}}{=}\underset{i\in V}{\sum}\underset{j\in V}{\sum}[\mathbf{z}_{i}^{n,\bar{w}}]_{j}\overset{\scriptsize{\mbox{Propn }\ref{prop:sparsity}}}{=}\underset{i\in V}{\sum}[\mathbf{z}_{i}^{n,\bar{w}}]_{i},\end{array}
\]
so $\sum_{i\in V}[\mathbf{z}_{i}^{n,\bar{w}}]_{i}$ is bounded. The
formulas \eqref{eq:norms-goto-inf} and \eqref{eq:z-i-hat} imply
that $\sum_{i\in V}[\hat{\mathbf{z}}_{i}]_{i}$ equals zero.

Let $x^{*}\in\mathbb{R}^{m}$ be any component of the primal optimal
solution $\mathbf{x}^{*}$. We now show that $[\hat{\mathbf{z}}_{i}]_{i}\in N_{\scriptsize{\cl(\dom(f_{i}))}}(x^{*})$.
On the one hand, we have 
\begin{equation}
\begin{array}{c}
\left\langle \underset{i\in V}{\sum}[\hat{\mathbf{z}}_{i}]_{i},\hat{x}-x^{*}\right\rangle =\left\langle 0,\hat{x}-x^{*}\right\rangle =0.\end{array}\label{eq:zero-sum-normals}
\end{equation}
On the other hand since $[\hat{\mathbf{z}}_{i}]_{i}\in N_{\scriptsize{\cl(\dom(f_{i}))}}(\hat{x})$
and $x^{*}\in\cap_{i\in V}\cl(\dom(f_{i}))$, we have 
\begin{equation}
\langle[\hat{\mathbf{z}}_{i}]_{i},\hat{x}-x^{*}\rangle\leq0.\label{eq:normal-formula}
\end{equation}
One can easily check that \eqref{eq:zero-sum-normals} and \eqref{eq:normal-formula}
implies that $\langle[\hat{\mathbf{z}}_{i}]_{i},\hat{x}-x^{*}\rangle=0$.
Next, for any $x\in\dom(f_{i})$, we have 
\[
\langle[\hat{\mathbf{z}}_{i}]_{i},x^{*}-x\rangle=\langle[\hat{\mathbf{z}}_{i}]_{i},x^{*}-\hat{x}\rangle+\langle[\hat{\mathbf{z}}_{i}]_{i},\hat{x}-x\rangle\overset{\eqref{eq:zero-sum-normals},\eqref{eq:normal-formula}}{\leq}0.
\]
Thus $[\hat{\mathbf{z}}_{i}]_{i}\in N_{\scriptsize{\cl(\dom(f_{i}))}}(x^{*})$.
This means that the constraint qualification \eqref{eq:CQ} fails
at $x^{*}$, which is a contradiction.
\end{proof}
\begin{rem}
\label{rem:BBL99}(Constraint qualifications on intersections of sets)
The assumption in Theorem \ref{thm:CQ-gives-bdd-iters} cannot be
easily weakened because the constraint qualification \eqref{eq:CQ}
is well known to be related to sensitivity analysis issues related
to the intersection of convex sets. For the convex case, we refer
to \cite{BBL99}. Such results have been extended to the nonconvex
case in, for example, \cite{Kruger_06} and the references mentioned
within.
\end{rem}

The theory on Dykstra's algorithm says that for any $w\in\{1,\dots,\bar{w}\}$,
$\lim_{n\to\infty}\mathbf{x}^{n,w}$ exists and is $x^{*}$, the minimizer
of the primal problem \eqref{eq:distrib-dyk-primal-pblm}. Hence under
the constraint qualification \eqref{eq:CQ}, the iterates $\{[\mathbf{z}_{i}^{n,w}]_{i}\}_{(n,w)}$
are bounded for all $i\in V$. 

\subsection{The $O(1/k)$ convergence}

Throughout this section, we make an assumption similar to Assumption
\ref{assu:lin-conv-assu}.
\begin{assumption}
\label{assu:O1-k-assu}For the problem \eqref{eq:general-framework},
suppose Assumption \ref{assu:lin-conv-assu}(2) and (3) hold. 
\end{assumption}

Recall the definition of $F_{S}(\cdot)$  in \eqref{eq:F-S-problem}.
Just like in Lemma \ref{lem:gamma-ineq}, we define $\mathbf{z}_{\alpha}^{0}\in[\mathbb{R}^{m}]^{|V|}$
and $\mathbf{z}_{\alpha}^{+}\in[\mathbb{R}^{m}]^{|V|}$ for all $\alpha\in V\cup\bar{E}$
so that $\mathbf{z}_{e}^{+}=\mathbf{z}_{e}^{0}$ for all $e\in\bar{E}$,
and $\mathbf{z}_{i}^{+}$ are defined through $\{\mathbf{z}_{\alpha}^{0}\}_{\alpha\in V\cup\bar{E}}$
by \eqref{eq:z-plus-minimizer}. 

Now, define $\mathbf{z}_{\alpha}^{++}\in[\mathbb{R}^{m}]^{|V|}$ so
that $\mathbf{z}_{i}^{++}=\mathbf{z}_{i}^{+}$ for all $i\in V$,
and for all $e\in\bar{E}$, $\mathbf{z}_{e}^{++}$ are defined so
that $\mathbf{v}_{H}^{++}\overset{\eqref{eq:v-H-def}}{=}\sum_{\alpha\in\bar{E}}\mathbf{z}_{\alpha}^{++}$
also satisfies 
\begin{equation}
\begin{array}{c}
\mathbf{v}_{H}^{++}=\underset{\mathbf{v}_{H}\in\mathbf{X}}{\arg\min}\frac{1}{2}\Big\|\underset{i\in V}{\sum}\mathbf{z}_{i}^{+}+\mathbf{v}_{H}-\bar{\mathbf{x}}\Big\|^{2}.\end{array}\label{eq:v-H-plus-plus}
\end{equation}
Note that $\mathbf{z}^{+}$ and $\mathbf{z}^{++}$ are defined by
a block coordinate minimization of the function $F_{S}(\cdot)$ starting
from the dual iterate $\mathbf{z}^{0}$. Consider $F_{V}:\mathbf{X}^{2}\to\mathbb{R}$
defined by 
\begin{equation}
\begin{array}{c}
F_{V}(\mathbf{z},\mathbf{v}_{H})=\underset{i\in V}{\overset{\phantom{i\in V}}{\sum}}\mathbf{f}_{i}^{*}([\mathbf{z}]_{i})+\delta_{D}^{*}(\mathbf{v}_{H})+\frac{1}{2}\|\mathbf{z}+\mathbf{v}_{H}-\bar{\mathbf{x}}\|^{2}.\end{array}\label{eq:F-V-form}
\end{equation}
Note that 
\begin{equation}
\begin{array}{c}
F_{S}(\mathbf{z})=F_{V}\Big(\underset{i\in V}{\overset{\phantom{i\in V}}{\sum}}\mathbf{z}_{i},\underset{e\in\bar{E}}{\overset{\phantom{e\in\bar{E}}}{\sum}}\mathbf{z}_{e}\Big).\end{array}\label{eq:F-S-F-V}
\end{equation}
The iterates $\mathbf{z}^{+}$ and $\mathbf{z}^{++}$ mentioned earlier
can be checked to be related to iterates produced by alternating minimization
on $F_{V}(\cdot)$. More precisely,\begin{subequations}\label{eq_m:alt_min}
\begin{eqnarray}
\begin{array}{c}
\underset{i\in V}{\overset{\phantom{i\in V}}{\sum}}\mathbf{z}_{i}^{+}\end{array} & = & \begin{array}{c}
\underset{\mathbf{z}\in\mathbf{X}}{\arg\min}\,\,F_{V}\Big(\mathbf{z},\underset{e\in\bar{E}}{\overset{\phantom{i\in V}}{\sum}}\mathbf{z}_{e}^{0}\Big)\end{array}\label{eq:first-term-argmin}\\
\begin{array}{c}
\mathbf{v}_{H}^{++}\end{array} & = & \begin{array}{c}
\underset{\mathbf{v}_{H}\in\mathbf{X}}{\arg\min}\,\,F_{V}\Big(\underset{i\in V}{\overset{\phantom{i\in V}}{\sum}}\mathbf{z}_{i}^{+},\mathbf{v}_{H}\Big).\end{array}\label{eq:v-H-argmin}
\end{eqnarray}
 \end{subequations} We first show the following:
\begin{lem}
Suppose $\mathbf{z}^{+}$ and $\mathbf{z}^{++}$ are defined via $\mathbf{z}^{0}$
being set to be $\mathbf{z}^{n,0}$. Then there is some constant $c_{1}>0$
such that 
\begin{equation}
\tilde{F}_{S}^{n,0}(\mathbf{z}^{n,0})-\tilde{F}_{S}^{n,\bar{w}}(\mathbf{z}^{n,\bar{w}})\geq c_{1}[\tilde{F}_{S}^{n,0}(\mathbf{z}^{n,0})-F_{S}(\mathbf{z}^{++})].\label{eq:decrease-satisfied}
\end{equation}
\end{lem}

\begin{proof}
Let $c_{3}<\frac{1}{2}$ be any constant. We divide into two cases.

\textbf{Case 1:} $\tilde{F}_{S}^{n,0}(\mathbf{z}^{n,0})-F_{S}(\mathbf{z}^{+})\geq c_{3}[\tilde{F}_{S}^{n,0}(\mathbf{z}^{n,0})-F_{S}(\mathbf{z}^{++})]$. 

By Lemma \ref{lem:subdiff-decrease}(2) and the definition of $\mathbf{z}^{+}$,
there is a constant $c_{2}>0$ such that $\tilde{F}_{S}^{n,0}(\mathbf{z}^{n,0})-\tilde{F}_{S}^{n,1}(\mathbf{z}^{n,1})\geq c_{2}[\tilde{F}_{S}^{n,0}(\mathbf{z}^{n,0})-F_{S}(\mathbf{z}^{+})]$.
We have 
\begin{eqnarray*}
 &  & \tilde{F}_{S}^{n,0}(\mathbf{z}^{n,0})-\tilde{F}_{S}^{n,\bar{w}}(\mathbf{z}^{n,\bar{w}})\overset{\scriptsize{\mbox{Thm \ref{thm:convergence}(i)}}}{\geq}\tilde{F}_{S}^{n,0}(\mathbf{z}^{n,0})-\tilde{F}_{S}^{n,1}(\mathbf{z}^{n,1})\\
 & \overset{\scriptsize{\mbox{Lem \ref{lem:subdiff-decrease}(2)}}}{\geq} & c_{2}[\tilde{F}_{S}^{n,0}(\mathbf{z}^{n,0})-F_{S}(\mathbf{z}^{+})]\overset{\scriptsize{\mbox{Case 1}}}{\geq}c_{2}c_{3}[\tilde{F}_{S}^{n,0}(\mathbf{z}^{n,0})-F_{S}(\mathbf{z}^{++})].
\end{eqnarray*}
\textbf{Case 2:} $\tilde{F}_{S}^{n,0}(\mathbf{z}^{n,0})-F_{S}(\mathbf{z}^{+})\leq c_{3}[\tilde{F}_{S}^{n,0}(\mathbf{z}^{n,0})-F_{S}(\mathbf{z}^{++})]$. 

From the inequality in \ref{thm:convergence}(i) and the definition
of $\mathbf{z}^{+}$, we have 
\begin{eqnarray}
 &  & \begin{array}{c}
c_{3}[\tilde{F}_{S}^{n,0}(\mathbf{z}^{n,0})-F_{S}(\mathbf{z}^{++})]\overset{\scriptsize{\mbox{Case 2}}}{\geq}\tilde{F}_{S}^{n,0}(\mathbf{z}^{n,0})-F_{S}(\mathbf{z}^{+})\phantom{\Big\|}\end{array}\label{eq:D-square-1}\\
 & \geq & \begin{array}{c}
F_{S}(\mathbf{z}^{n,0})-F_{S}(\mathbf{z}^{+})\overset{\eqref{eq:F-S-F-V}}{=}F_{V}\Big(\underset{i\in V}{\overset{\phantom{i\in V}}{\sum}}\mathbf{z}_{i}^{0},\underset{e\in\bar{E}}{\overset{\phantom{e\in\bar{E}}}{\sum}}\mathbf{z}_{e}^{0}\Big)-F_{V}\Big(\underset{i\in V}{\overset{\phantom{i\in V}}{\sum}}\mathbf{z}_{i}^{+},\underset{e\in\bar{E}}{\overset{\phantom{e\in\bar{E}}}{\sum}}\mathbf{z}_{e}^{0}\Big)\end{array}\nonumber \\
 & \overset{\eqref{eq:F-V-form},\eqref{eq:first-term-argmin}}{\geq} & \begin{array}{c}
\frac{1}{2}\Big\|\underset{\alpha\in V\cup\bar{E}}{\overset{\phantom{\alpha\in V\cup E}}{\sum}}\mathbf{z}_{\alpha}^{n,0}-\underset{\alpha\in V\cup\bar{E}}{\sum}\mathbf{z}_{\alpha}^{+}\Big\|^{2}.\end{array}\nonumber 
\end{eqnarray}
Since $\mathbf{z}^{++}$ was defined so that $\bar{\mathbf{x}}-\sum_{\alpha\in V\cup\bar{E}}\mathbf{z}_{\alpha}^{++}$
is the projection of $\bar{\mathbf{x}}-\sum_{\alpha\in V\cup\bar{E}}\mathbf{z}_{\alpha}^{+}$
onto $D$ as defined in \eqref{eq:D-formula-a}, we have 
\begin{eqnarray}
\begin{array}{c}
\frac{1}{2}d\Big(\bar{\mathbf{x}}-\underset{\alpha\in V\cup E}{\sum}\mathbf{z}_{\alpha}^{+},D\Big)^{2}\end{array} & = & \begin{array}{c}
F_{S}(\mathbf{z}^{+})-F_{S}(\mathbf{z}^{++})\end{array}\label{eq:D-square-2}\\
 & = & \begin{array}{c}
\tilde{F}_{S}^{n,0}(\mathbf{z}^{n,0})-F_{S}(\mathbf{z}^{++})-[\tilde{F}_{S}^{n,0}(\mathbf{z}^{n,0})-F_{S}(\mathbf{z}^{+})]\end{array}\nonumber \\
 & \overset{\scriptsize{\mbox{Case 2}}}{\geq} & \begin{array}{c}
(1-c_{3})[\tilde{F}_{S}^{n,0}(\mathbf{z}^{n,0})-F_{S}(\mathbf{z}^{++})].\end{array}\nonumber 
\end{eqnarray}
So 
\begin{eqnarray}
 &  & \begin{array}{c}
d\Big(\bar{\mathbf{x}}-\underset{\alpha\in V\cup\bar{E}}{\overset{\phantom{\alpha\in V\cup\bar{E}}}{\sum}}\mathbf{z}_{\alpha}^{n,0},D\Big)\end{array}\label{eq:dist-comp-to-sqrt}\\
 & \geq & \begin{array}{c}
d\Big(\bar{\mathbf{x}}-\underset{\alpha\in V\cup\bar{E}}{\overset{\phantom{\alpha\in V\cup\bar{E}}}{\sum}}\mathbf{z}_{\alpha}^{+},D\Big)-\Big\|\underset{\alpha\in V\cup\bar{E}}{\sum}\mathbf{z}_{\alpha}^{+}-\underset{\alpha\in V\cup\bar{E}}{\sum}\mathbf{z}_{\alpha}^{n,0}\Big\|\end{array}\nonumber \\
 & \overset{\eqref{eq:D-square-1},\eqref{eq:D-square-2}}{\geq} & \begin{array}{c}
\left[\sqrt{1-c_{3}}-\sqrt{c_{3}}\right]\sqrt{2[\tilde{F}_{S}^{n,0}(\mathbf{z}^{n,0})-F_{S}(\mathbf{z}^{++})]}.\end{array}\nonumber 
\end{eqnarray}
Since $c_{3}<\frac{1}{2}$, $\sqrt{1-c_{3}}-\sqrt{c_{3}}>0$. By linear
regularity of $D=\cap_{\alpha\in\bar{E}}H_{\alpha}$ as in Definition
\ref{def:connects} (since $H_{\alpha}$ are all linear subspaces),
there is some $\kappa>0$ such that $d\left(\bar{\mathbf{x}}-\sum_{\alpha\in V\cup\bar{E}}\mathbf{z}_{\alpha}^{n,0},D\right)\leq\kappa\max_{\alpha\in\bar{E}_{n}}d(\bar{\mathbf{x}}-\sum_{\alpha\in V\cup\bar{E}_{n}}\mathbf{z}_{\alpha}^{n,0},H_{\alpha})$.
Let the right hand side be attained by some $\alpha^{*}\in\bar{E}_{n}$.
There is some $w^{*}\in\{1,\dots,\bar{w}\}$ such that $\alpha^{*}\in S_{n,w^{*}}$.
Now, since $\bar{\mathbf{x}}-\sum_{\alpha\in V\cup E}\mathbf{z}_{\alpha}^{n,w^{*}}\in H_{\alpha^{*}}$,
we have 
\begin{eqnarray*}
 &  & \begin{array}{c}
\tilde{F}_{S}^{n,0}(\mathbf{z}^{n,0})-\tilde{F}_{S}^{n,\bar{w}}(\mathbf{z}^{n,\bar{w}})\end{array}\\
 & \overset{\scriptsize{\mbox{Thm \ref{thm:convergence}(i)}}}{\geq} & \begin{array}{c}
\tilde{F}_{S}^{n,0}(\mathbf{z}^{n,0})-\tilde{F}_{S}^{n,w^{*}}(\mathbf{z}^{n,w^{*}})\end{array}\\
 & \overset{\scriptsize{\mbox{Thm \ref{thm:convergence}(i)}}}{\geq} & \begin{array}{c}
\underset{w=1}{\overset{w^{*}}{\sum}}\frac{1}{2}\Big\|\underset{\alpha\in V\cup\bar{E}}{\overset{\phantom{\alpha\in V\cup\bar{E}}}{\sum}}\mathbf{z}^{n,w}-\underset{\alpha\in V\cup\bar{E}}{\overset{\phantom{\alpha\in V\cup\bar{E}}}{\sum}}\mathbf{z}^{n,w-1}\Big\|^{2}\end{array}\\
 & \geq & \begin{array}{c}
\frac{1}{w^{*}}\frac{1}{2}\Big\|\underset{\alpha\in V\cup\bar{E}}{\overset{\phantom{\alpha\in V\cup\bar{E}}}{\sum}}\mathbf{z}^{n,w^{*}}-\underset{\alpha\in V\cup\bar{E}}{\overset{\phantom{\alpha\in V\cup\bar{E}}}{\sum}}\mathbf{z}^{n,0}\Big\|^{2}\end{array}\\
 & \overset{\bar{\mathbf{x}}-\sum_{\alpha\in V\cup\bar{E}}\mathbf{z}_{\alpha}^{n,w^{*}}\in H_{\alpha^{*}}}{\geq} & \begin{array}{c}
\frac{1}{2\bar{w}}d\Big(\bar{\mathbf{x}}-\underset{\alpha\in V\cup\bar{E}}{\overset{\phantom{\alpha\in V\cup\bar{E}}}{\sum}}\mathbf{z}^{n,0},D\Big)^{2}\end{array}\\
 & \overset{\scriptsize{\eqref{eq:dist-comp-to-sqrt}}}{\geq} & \begin{array}{c}
\frac{1}{\bar{w}}\left[\sqrt{1-c_{3}}-\sqrt{c_{3}}\right]^{2}[\tilde{F}_{S}^{n,0}(\mathbf{z}^{n,0})-F_{S}(\mathbf{z}^{++})].\end{array}
\end{eqnarray*}
This concludes the proof. 
\end{proof}
\begin{thm}
Let $h^{n}=\tilde{F}_{S}^{n,0}(\mathbf{z}^{n,0})-F_{S}(\mathbf{z}^{*})$.
The values $\{h^{n}\}_{n=1}^{\infty}$ converge to zero at the rate
of $O(1/n)$. 
\end{thm}

\begin{proof}
From the form of $F_{V}(\cdot)$ in \eqref{eq:F-V-form}, we have\begin{subequations}
\begin{eqnarray*}
0 & \overset{\eqref{eq:first-term-argmin}}{\in} & \begin{array}{c}
\partial\Big[\underset{i\in V}{\overset{\phantom{i\in V}}{\sum}}\mathbf{f}_{i}(\cdot)\Big]\Big(\underset{i\in V}{\overset{\phantom{i\in V}}{\sum}}\mathbf{z}_{i}^{+}\Big)+\underset{i\in V}{\overset{\phantom{i\in V}}{\sum}}\mathbf{z}_{i}^{+}+\mathbf{v}_{H}^{0}-\bar{\mathbf{x}}\end{array}\\
0 & \overset{\eqref{eq:v-H-argmin}}{\in} & \begin{array}{c}
\partial\delta_{D}^{*}(\mathbf{v}_{H}^{++})+\underset{i\in V}{\overset{\phantom{i\in V}}{\sum}}\mathbf{z}_{i}^{+}+\mathbf{v}_{H}^{++}-\bar{\mathbf{x}}.\end{array}
\end{eqnarray*}
\end{subequations}We then have $(\mathbf{v}_{H}^{++}-\mathbf{v}_{H}^{0},0)\in\partial F_{V}(\sum_{i\in V}\mathbf{z}_{i}^{+},\mathbf{v}_{H}^{++})$.
Let an optimal solution of $F_{V}(\cdot)$ be $(\sum_{i\in V}\mathbf{z}^{*},\mathbf{v}_{H}^{*})$.
Then 
\begin{eqnarray}
\begin{array}{c}
F_{S}(\mathbf{z}^{++})-F_{S}(\mathbf{z}^{*})\end{array} & \overset{\eqref{eq:F-S-F-V}}{=} & \begin{array}{c}
F_{V}\Big(\underset{i\in V}{\overset{\phantom{i\in V}}{\sum}}\mathbf{z}_{i}^{+},\mathbf{v}_{H}^{++}\Big)-F_{V}\Big(\underset{i\in V}{\overset{\phantom{i\in V}}{\sum}}\mathbf{z}^{*},\mathbf{v}_{H}^{*}\Big)\end{array}\label{eq:F-V-chain}\\
 & \leq & \begin{array}{c}
-\Big\langle(\mathbf{v}_{H}^{++}-\mathbf{v}_{H}^{0},0),\Big(\underset{i\in V}{\overset{\phantom{i\in V}}{\sum}}\mathbf{z}^{*},\mathbf{v}_{H}^{*}\Big)-\Big(\underset{i\in V}{\overset{\phantom{i\in V}}{\sum}}\mathbf{z}_{i}^{+},\mathbf{v}_{H}^{++}\Big)\Big\rangle\end{array}\nonumber \\
 & \leq & \begin{array}{c}
\|\mathbf{v}_{H}^{++}-\mathbf{v}_{H}^{0}\|\Big\|\underset{i\in V}{\overset{\phantom{i\in V}}{\sum}}\mathbf{z}^{*}-\underset{i\in V}{\overset{\phantom{i\in V}}{\sum}}\mathbf{z}_{i}^{+}\Big\|.\end{array}\nonumber 
\end{eqnarray}
Since the $\mathbf{z}_{i}^{+}$ would be bounded if the constraint
qualification \eqref{eq:CQ} is satisfied, there is some $C>0$ such
that $\left\Vert \sum_{i\in V}\mathbf{z}^{*}-\sum_{i\in V}\mathbf{z}_{i}^{+}\right\Vert \leq C$.
We also have 
\begin{eqnarray}
 &  & \begin{array}{c}
\|\mathbf{v}_{H}^{++}-\mathbf{v}_{H}^{0}\|\end{array}\overset{\eqref{eq:v-H-argmin}}{\leq}\begin{array}{c}
\sqrt{2\left[F_{V}\left(\underset{i\in V}{\overset{\phantom{i\in V}}{\sum}}\mathbf{z}_{i}^{+},\mathbf{v}_{H}^{++}\right)-F_{V}\left(\underset{i\in V}{\overset{\phantom{i\in V}}{\sum}}\mathbf{z}_{i}^{+},\mathbf{v}_{H}^{0}\right)\right]}\end{array}\nonumber \\
 & = & \begin{array}{c}
\sqrt{2[F_{S}(\mathbf{z}^{++})-F_{S}(\mathbf{z}^{+})]}\end{array}\leq\begin{array}{c}
\sqrt{2[F_{S}(\mathbf{z}^{++})-F_{S}(\mathbf{z}^{0})]}.\end{array}\label{eq:F-V-chain-2}
\end{eqnarray}
Hence we have 
\begin{equation}
F_{S}(\mathbf{z}^{++})-F_{S}(\mathbf{z}^{*})\overset{\eqref{eq:F-V-chain}}{\leq}C\|\mathbf{v}_{H}^{++}-\mathbf{v}_{H}^{0}\|\overset{\eqref{eq:F-V-chain-2}}{\leq}\sqrt{2}C\sqrt{F_{S}(\mathbf{z}^{++})-F_{S}(\mathbf{z}^{0})},\label{eq:recur-up}
\end{equation}
so 
\begin{eqnarray}
\begin{array}{c}
\tilde{F}_{S}^{n,0}(\mathbf{z}^{n,0})-F_{S}(\mathbf{z}^{*})\end{array} & \geq & \begin{array}{c}
F_{S}(\mathbf{z}^{0})-F_{S}(\mathbf{z}^{*})\end{array}\label{eq:recur}\\
 & \overset{\eqref{eq:recur-up}}{\geq} & \begin{array}{c}
[F_{S}(\mathbf{z}^{++})-F_{S}(\mathbf{z}^{*})]+\frac{1}{2C^{2}}[F_{S}(\mathbf{z}^{++})-F_{S}(\mathbf{z}^{*})]^{2}.\end{array}\nonumber 
\end{eqnarray}
Let $h^{++}:=F_{S}(\mathbf{z}^{++})-F_{S}(\mathbf{z}^{*})$. We have
\begin{equation}
\begin{array}{c}
\frac{1}{h^{n+1}}-\frac{1}{h^{n}}=\frac{h^{n}-h^{n+1}}{h^{n+1}h^{n}}\geq\frac{h^{n}-h^{n+1}}{[h^{n}]^{2}}\overset{\eqref{eq:decrease-satisfied}}{\geq}\frac{c_{1}[h^{n}-h^{++}]}{[h^{n}]^{2}}\overset{\eqref{eq:recur}}{\geq}\frac{c_{1}}{2C^{2}}\frac{[h^{++}]^{2}}{[h^{n}]^{2}}.\end{array}\label{eq:recip-1-line}
\end{equation}
Also, 
\[
\begin{array}{c}
\frac{1}{h^{n+1}}-\frac{1}{h^{n}}\overset{\eqref{eq:recip-1-line}}{\geq}\frac{c_{1}[h^{n}-h^{++}]}{[h^{n}]^{2}}\geq\frac{c_{1}}{h^{n}}\left(1-\frac{h^{++}}{h^{n}}\right)\geq\frac{c_{1}}{h^{1}}\left(1-\frac{h^{++}}{h^{n}}\right).\end{array}
\]
We can check from simple calculus that there is a constant $c_{4}>0$
such that $\max\{\frac{c_{1}}{2C^{2}}\left(\frac{h^{++}}{h^{n}}\right)^{2},\frac{c_{1}}{h^{1}}\left(1-\frac{h^{++}}{h^{n}}\right)\}>c_{4}$,
which implies that $\frac{1}{h^{n+1}}-\frac{1}{h^{n}}>c_{4}$. This
implies the $O(1/n)$ convergence of $\{h^{n}\}_{n=1}^{\infty}$ as
needed. 
\end{proof}

\section{\label{sec:Sublin-conv}$O(1/k^{1/3})$ convergence when some functions
are subdifferentiable but not smooth}

In this section, we prove the $O(1/n^{1/3})$ convergence rate in
the general case. Here, the sets $V$, $\bar{E}$ and $\mathbf{X}$
need not take the form in Example \ref{exa:distrib-dyk}. 

 For $i\in V_{4}$, let $\hat{\mathbf{z}}_{i}^{n,p(n,i)}\in\mathbf{X}$
be the minimizer of 
\begin{equation}
\min_{\hat{\mathbf{z}}_{i}\in\mathbf{X}}\frac{1}{2}\bigg\|\hat{\mathbf{z}}_{i}-\bigg[\bar{\mathbf{x}}-\sum_{\beta\neq i}\mathbf{z}_{\beta}^{n,p(n,i)}\bigg]\bigg\|^{2}+\mathbf{f}_{i}^{*}(\hat{\mathbf{z}}_{i}).\label{eq:hat-dual}
\end{equation}
Note that by \eqref{eq:Dykstra-min-subpblm}, $\mathbf{z}_{i}^{n,p(n,i)}$
is the minimizer of a similar problem as \eqref{eq:hat-dual} but
with $\mathbf{f}_{i}^{*}(\cdot)$ replaced by $\mathbf{f}_{i,n,p(n,i)}^{*}(\cdot)$.
So $\hat{\mathbf{z}}_{i}^{n,p(n,i)}$ is distinct from $\mathbf{z}_{i}^{n,p(n,i)}$.
We also define $\hat{\mathbf{z}}_{i}^{n,\bar{w}}$ to be $\hat{\mathbf{z}}_{i}^{n,\bar{w}}=\hat{\mathbf{z}}_{i}^{n,p(n,i)}$.
Since $\mathbf{f}_{i}(\cdot)$ depends only on the $i$-th coordinate
of its input $\mathbf{x}$, one can check that $\hat{\mathbf{z}}_{i}\in\partial\mathbf{f}_{i}(\mathbf{x})$
for some $\mathbf{x}\in\mathbf{X}$, which implies that $[\hat{\mathbf{z}}_{i}]_{j}=0$
for all $j\neq i$. The problem \eqref{eq:hat-dual} is equivalent
to the problem of finding $\hat{z}_{i}^{n,p(n,i)}\in X_{i}$ that
minimizes 
\begin{equation}
\min_{\hat{z}_{i}\in X_{i}}\frac{1}{2}\bigg\|\hat{z}_{i}-\underbrace{\bigg([\bar{\mathbf{x}}]_{i}-\sum_{\beta\neq i}[\mathbf{z}_{\beta}^{n,p(n,i)}]_{i}\bigg)}_{p_{i}}\bigg\|^{2}+f_{i}^{*}(\hat{z}_{i}),\label{eq:small-sub-dual}
\end{equation}
where $\hat{\mathbf{z}}_{i}^{n,p(n,i)}$ and $\hat{z}_{i}^{n,p(n,i)}$
are related by the formula $[\mathbf{\hat{z}}_{i}^{n,p(n,i)}]_{j}=0$
if $j\neq i$, and $[\hat{\mathbf{z}}_{i}^{n,p(n,i)}]_{i}=\hat{z}_{i}^{n,p(n,i)}$.
Let $p_{i}$ be the prox center as marked in \eqref{eq:small-sub-dual}.
The dual of \eqref{eq:small-sub-dual} is, up to a constant and a
change of sign, 
\[
\begin{array}{c}
\underset{\hat{x}_{i}\in X_{i}}{\min}\frac{1}{2}\|\hat{x}_{i}-p_{i}\|^{2}+f_{i}(\hat{x}_{i}).\end{array}
\]
(A more accurate primal-dual pair is \eqref{eq:lemma-primal} and
\eqref{eq:lemma-dual}, but this form is equivalent up to a sign change
and constant.) By the Moreau decomposition theorem, the\textbf{ $\hat{x}_{i}^{n,p(n,i)}$,
}$\hat{z}_{i}^{n,p(n,i)}$, $[\mathbf{x}^{n,p(n,i)}]_{i}$ and $[\mathbf{z}_{i}^{n,p(n,i)}]_{i}$
satisfy 
\begin{equation}
\hat{x}_{i}^{n,p(n,i)}+\hat{z}_{i}^{n,p(n,i)}=p_{i}\mbox{ and }[\mathbf{x}^{n,p(n,i)}]_{i}+[\mathbf{z}_{i}^{n,p(n,i)}]_{i}=p_{i}.\label{eq:Moreau-conseq}
\end{equation}
From optimality conditions of \eqref{eq:hat-dual} and \eqref{eq:Dykstra-min-subpblm},
we have 
\begin{eqnarray}
 &  & \begin{array}{c}
0\overset{\eqref{eq:hat-dual}}{\in}\underset{\beta\neq i}{\overset{\phantom{\beta\neq i}}{\sum}}\mathbf{z}_{\beta}^{n,p(n,i)}+\hat{\mathbf{z}}_{i}^{n,p(n,i)}-\bar{\mathbf{x}}+\partial\mathbf{f}_{i}^{*}(\hat{\mathbf{z}}_{i}^{n,p(n,i)})\mbox{ for all }i\in V_{4},\mbox{ and }\end{array}\label{eq:subdiff-at-i-2nd}\\
 &  & \begin{array}{c}
0\overset{\eqref{eq:Dykstra-min-subpblm}}{\in}\underset{\beta\in V\cup\bar{E}}{\overset{\phantom{\beta\in V\cup\bar{E}}}{\sum}}\mathbf{z}_{\beta}^{n,p(n,\alpha)}-\bar{\mathbf{x}}+\partial\mathbf{f}_{\alpha,n,p(n,\alpha)}^{*}(\mathbf{z}_{\alpha}^{n,p(n,\alpha)})\mbox{ for all }\alpha\in V\cup\bar{E}.\end{array}\label{eq:aubdiff-at-i}
\end{eqnarray}
If $\alpha\in V_{1}\cup V_{2}\cup V_{3}\cup\bar{E}$, then $\hat{\mathbf{z}}_{i}^{n,p(n,i)}=\mathbf{z}_{i}^{n,p(n,i)}$.
Next, define $\Delta z_{i}\in X_{i}$ by 
\begin{equation}
\Delta z_{i}:=[\hat{\mathbf{z}}_{i}^{n,p(n,i)}]_{i}-[\mathbf{z}_{i}^{n,p(n,i)}]_{i}\overset{\eqref{eq:Moreau-conseq}}{=}[\mathbf{x}^{n,p(n,i)}]_{i}-\hat{x}_{i}^{n,p(n,i)}\in X_{i}.\label{eq:def-delta-z-i}
\end{equation}
We have 
\begin{eqnarray}
 &  & \begin{array}{c}
\langle[\mathbf{x}^{n,p(n,i)}]_{i},[\mathbf{z}_{i}^{n,p(n,i)}]_{i}\rangle-\langle\hat{x}_{i}^{n,p(n,i)},[\hat{\mathbf{z}}_{i}^{n,p(n,i)}]_{i}\rangle\end{array}\label{eq:hat-and-no-hat}\\
 & \overset{\eqref{eq:def-delta-z-i}}{=} & \begin{array}{c}
\langle[\mathbf{x}^{n,p(n,i)}]_{i}-[\mathbf{z}_{i}^{n,p(n,i)}]_{i},\Delta z_{i}\rangle+\frac{1}{2}\|\Delta z_{i}\|^{2}.\end{array}\nonumber 
\end{eqnarray}
We also have that $\hat{z}_{i}^{n,p(n,i)}\in\partial f_{i}(\hat{x}_{i}^{n,p(n,i)})$.
Recall \eqref{eq:h-def}. At the point $\hat{\mathbf{z}}^{n,\bar{w}}\in\mathbf{X}^{|V\cup\bar{E}|}$,
due to the separability of the non-quadratic term, we have, for each
$\alpha\in V\cup\bar{E}$, the partial subdifferential of $(-F)$
in the $\alpha$-th coordinate is 
\begin{eqnarray}
\partial(-F)_{\alpha}(\hat{\mathbf{z}}^{n,\bar{w}}) & \overset{\eqref{eq:h-def}}{=} & \sum_{\beta\in V\cup\bar{E}}\hat{\mathbf{z}}_{\beta}^{n,\bar{w}}-\bar{\mathbf{x}}+\partial\mathbf{f}_{\alpha}^{*}(\hat{\mathbf{z}}_{\alpha}^{n,p(n,\alpha)})\label{eq:partial-subdiff}\\
 & \overset{\eqref{eq:subdiff-at-i-2nd}}{\ni} & \underbrace{\sum_{\beta\neq\alpha}\hat{\mathbf{z}}_{\beta}^{n,\bar{w}}}_{\hat{\mathbf{t}}_{\alpha}}-\underbrace{\sum_{\beta\neq\alpha}\mathbf{z}_{\beta}^{n,p(n,\alpha)}}_{\mathbf{t}_{\alpha}}.\nonumber 
\end{eqnarray}
Let $\hat{\mathbf{t}}_{\alpha}$ and $\mathbf{t}_{\alpha}$ be as
marked above. Let $\mathbf{s}\in\mathbf{X}^{|V\cup\bar{E}|}$ be defined
by $\mathbf{s}_{\alpha}:=\hat{\mathbf{t}}_{\alpha}-\mathbf{t}_{\alpha}$.
In view of \eqref{eq:partial-subdiff}, we have $\mathbf{s}\in\partial(-F)(\hat{\mathbf{z}}^{n,\bar{w}})$.
Define $\bar{\mathbf{t}}_{\alpha}\in\mathbf{X}$ to be 
\begin{equation}
\begin{array}{c}
\bar{\mathbf{t}}_{\alpha}:=\underset{\beta\neq\alpha}{\overset{\phantom{\beta\neq\alpha}}{\sum}}\mathbf{z}_{\beta}^{n,\bar{w}}.\end{array}\label{eq:big-t-bar}
\end{equation}
Let $\mathbf{z}^{*}$ be an optimizer to \eqref{eq:general-dual},
which we assume to exist. Since $\mathbf{s}\in\partial(-F)(\hat{\mathbf{z}}^{n,\bar{w}})$,
we have 
\begin{eqnarray}
 &  & \begin{array}{c}
-F(\hat{\mathbf{z}}^{n,\bar{w}})-[-F(\mathbf{z}^{*})]\end{array}\label{eq:subdiff-for-h}\\
 & \overset{\mathbf{s}\in\partial(-F)(\hat{\mathbf{z}}^{n,\bar{w}})}{\leq} & \begin{array}{c}
-\langle\mathbf{s},\mathbf{z}^{*}-\hat{\mathbf{z}}^{n,\bar{w}}\rangle=\underset{\alpha\in V\cup\bar{E}}{\overset{\phantom{\alpha\in V\cup\bar{E}}}{\sum}}-\langle\hat{\mathbf{t}}_{\alpha}-\mathbf{t}_{\alpha},\mathbf{z}_{\alpha}^{*}-\hat{\mathbf{z}}_{\alpha}^{n,\bar{w}}\rangle\end{array}\nonumber \\
 & \leq & \begin{array}{c}
\underset{\alpha\in V\cup\bar{E}}{\overset{\phantom{\alpha\in V\cup\bar{E}}}{\sum}}\big[[\|\bar{\mathbf{t}}_{\alpha}-\hat{\mathbf{t}}_{\alpha}\|+\|\mathbf{t}_{\alpha}-\bar{\mathbf{t}}_{\alpha}\|]\|\mathbf{z}_{\alpha}^{*}-\hat{\mathbf{z}}_{\alpha}^{n,\bar{w}}\|\big].\end{array}\nonumber 
\end{eqnarray}
Since $\mathbf{x}^{n,w}$ equals $\bar{\mathbf{x}}-\mathbf{v}_{A}^{n,w}$
is bounded by Theorem \ref{thm:convergence}(ii), for all $i\in V_{4}$,
the subdifferentiable function $f_{i}(\cdot)$ is Lipschitz with some
constant $L_{i}$ in the domain of interest. Define $\hat{\mathbf{v}}_{A}^{n,w}\in\mathbf{X}$
like in \eqref{eq_m:all_acronyms} to be 
\begin{equation}
\begin{array}{c}
\hat{\mathbf{v}}_{A}^{n,w}:=\underset{\alpha\in V\cup\bar{E}}{\overset{\phantom{\alpha\in V\cup\bar{E}}}{\sum}}\hat{\mathbf{z}}_{i}^{n,w}.\end{array}\label{eq:v-hat}
\end{equation}
From the fact that $\hat{\mathbf{z}}_{i}^{n,w}=\mathbf{z}_{i}^{n,w}$
for all $i\in V_{1}\cup V_{2}\cup V_{3}$, we have 
\begin{eqnarray}
 &  & -[F^{n,\bar{w}}(\mathbf{z}^{n,\bar{w}})-F(\hat{\mathbf{z}}^{n,\bar{w}})]\label{eq:subdiff-part-2}\\
 & \overset{\eqref{eq:h-def},\eqref{eq:h-def-1}}{=} & \sum_{i\in V_{4}}[\mathbf{f}_{i,n,\bar{w}}^{*}(\mathbf{z}_{i}^{n,\bar{w}})-\mathbf{f}_{i}^{*}(\hat{\mathbf{z}}_{i}^{n,\bar{w}})]+\frac{1}{2}\|\bar{\mathbf{x}}-\sum_{\beta\in V\cup\bar{E}}\mathbf{z}_{\beta}^{n,\bar{w}}\|^{2}-\frac{1}{2}\|\bar{\mathbf{x}}-\sum_{\beta\in V\cup\bar{E}}\hat{\mathbf{z}}_{\beta}^{n,\bar{w}}\|^{2}\nonumber \\
 & \overset{\eqref{eq:v-hat},\eqref{eq_m:all_acronyms}}{=} & \sum_{i\in V_{4}}[f_{i,n,\bar{w}}^{*}([\mathbf{z}_{i}^{n,\bar{w}}]_{i})-f_{i}^{*}([\hat{\mathbf{z}}_{i}^{n,\bar{w}}]_{i})]+\frac{1}{2}\|\bar{\mathbf{x}}-\mathbf{v}_{A}^{n,\bar{w}}\|^{2}-\frac{1}{2}\|\bar{\mathbf{x}}-\hat{\mathbf{v}}_{A}^{n,\bar{w}}\|^{2}\nonumber \\
 & \overset{\eqref{eq:stagnant-indices},\eqref{eq:x-from-v-A}}{=} & \sum_{i\in V_{4}}\big[\langle[\mathbf{x}^{n,p(n,i)}]_{i},[\mathbf{z}_{i}^{n,p(n,i)}]_{i}\rangle-f_{i,n,p(n,i)}([\mathbf{x}^{n,p(n,i)}]_{i})-\langle\hat{x}_{i}^{n,p(n,i)},[\hat{\mathbf{z}}_{i}^{n,p(n,i)}]_{i}\rangle+f_{i}(\hat{x}_{i}^{n,p(n,i)})\big]\nonumber \\
 &  & -\langle(\bar{\mathbf{x}}-\mathbf{v}_{A}^{n,\bar{w}}),(\mathbf{v}_{A}^{n,\bar{w}}-\hat{\mathbf{v}}_{A}^{n,\bar{w}})\rangle-\frac{1}{2}\|\mathbf{v}_{A}^{n,\bar{w}}-\hat{\mathbf{v}}_{A}^{n,\bar{w}}\|^{2}\nonumber \\
 & \overset{\eqref{eq:hat-and-no-hat}}{=} & \sum_{i\in V_{4}}\big[\langle[\mathbf{x}^{n,p(n,i)}]_{i}-[\mathbf{z}_{i}^{n,p(n,i)}]_{i},\Delta z_{i}\rangle+\frac{1}{2}\|\Delta z_{i}\|^{2}+[f_{i}(\hat{x}_{i}^{n,p(n,i)})-f_{i,n,p(n,i)}([\mathbf{x}^{n,p(n,i)}]_{i})]\big]\nonumber \\
 &  & -\langle(\bar{\mathbf{x}}-\mathbf{v}_{A}^{n,\bar{w}}),(\mathbf{v}_{A}^{n,\bar{w}}-\hat{\mathbf{v}}_{A}^{n,\bar{w}})\rangle-\frac{1}{2}\|\sum_{i\in V_{4}}[\mathbf{z}_{i}^{n,\bar{w}}-\hat{\mathbf{z}}_{i}^{n,\bar{w}}]\|^{2}\nonumber \\
 & \overset{\eqref{eq:def-delta-z-i}}{\leq} & \sum_{i\in V_{4}}\big[\langle[\mathbf{x}^{n,p(n,i)}]_{i}-[\mathbf{z}_{i}^{n,p(n,i)}]_{i},\Delta z_{i}\rangle+[f_{i}([\mathbf{x}^{n,p(n,i)}]_{i})-f_{i,n,p(n,i)}([\mathbf{x}^{n,p(n,i)}]_{i})+L_{i}\|\Delta z_{i}\|]\big]\nonumber \\
 &  & +\|\bar{\mathbf{x}}-\mathbf{v}_{A}^{n,\bar{w}}\|\|\mathbf{v}_{A}^{n,\bar{w}}-\hat{\mathbf{v}}_{A}^{n,\bar{w}}\|.\nonumber 
\end{eqnarray}
Define $\Delta f_{i}\in\mathbb{R}$ to be 
\begin{equation}
\Delta f_{i}:=f_{i}([\mathbf{x}^{n,p(n,i)}]_{i})-f_{i,n,p(n,i)}([\mathbf{x}^{n,p(n,i)}]_{i}).\label{eq:def-delta-f-i}
\end{equation}
Since $f_{i,n,w}(\cdot)\leq f_{i}(\cdot)$, we have $\Delta f_{i}\geq0$. 
\begin{lem}
\label{lem:grad-sqrt-bdd}Recall the formulas of $\Delta z_{i}$ in
\eqref{eq:def-z-i} and $\Delta f_{i}$. For all $i\in V_{4}$, we
have 
\begin{equation}
\|\Delta z_{i}\|\leq\sqrt{\Delta f_{i}}.\label{eq:gradient-sqrt-bdd}
\end{equation}
\end{lem}

\begin{proof}
Since $[\mathbf{x}^{n,p(n,i)}]_{i}$ is the minimizer of $f_{i,n,p(n,i)}(\cdot)+\frac{1}{2}\|\cdot-p_{i}\|^{2}$,
where $p_{i}$ is as in \eqref{eq:small-sub-dual}, and $\hat{x}_{i}^{n,p(n,i)}$
is the minimizer of $f_{i}(\cdot)+\frac{1}{2}\|\cdot-p_{i}\|^{2}$,
we have 
\begin{eqnarray}
 &  & \begin{array}{c}
f_{i,n,p(n,i)}([\mathbf{x}^{n,p(n,i)}]_{i})+\frac{1}{2}\|[\mathbf{x}^{n,p(n,i)}]_{i}-p_{i}\|^{2}+\|[\mathbf{x}^{n,p(n,i)}]_{i}-\hat{x}_{i}^{n,p(n,i)}\|^{2}\end{array}\nonumber \\
 & \leq & \begin{array}{c}
f_{i,n,p(n,i)}(\hat{x}_{i}^{n,p(n,i)})+\frac{1}{2}\|\hat{x}_{i}^{n,p(n,i)}-p_{i}\|^{2}+\frac{1}{2}\|[\mathbf{x}^{n,p(n,i)}]_{i}-\hat{x}_{i}^{n,p(n,i)}\|^{2}\end{array}\nonumber \\
 & \leq & \begin{array}{c}
f_{i}(\hat{x}_{i}^{n,p(n,i)})+\frac{1}{2}\|\hat{x}_{i}^{n,p(n,i)}-p_{i}\|^{2}+\frac{1}{2}\|[\mathbf{x}^{n,p(n,i)}]_{i}-\hat{x}_{i}^{n,p(n,i)}\|^{2}\end{array}\nonumber \\
 & \leq & \begin{array}{c}
f_{i}([\mathbf{x}^{n,p(n,i)}]_{i})+\frac{1}{2}\|[\mathbf{x}^{n,p(n,i)}]_{i}-p_{i}\|^{2}.\end{array}\label{eq:norm-and-sqrt}
\end{eqnarray}
Rearranging inequality \eqref{eq:norm-and-sqrt} and using \eqref{eq:def-delta-z-i}
gives us what we need. 
\end{proof}
In view of Lemma \ref{lem:grad-sqrt-bdd} and the fact that $\mathbf{z}_{\alpha}^{n,p(n,\alpha)}=\hat{\mathbf{z}}_{\alpha}^{n,p(n,\alpha)}$
for all $\alpha\in V_{1}\cup V_{2}\cup V_{3}\cup\bar{E}$, we have
\begin{subequations}\label{eq_m:two-sum-bdds}
\begin{eqnarray}
 &  & \begin{array}{c}
\bar{\mathbf{t}}_{\alpha}-\hat{\mathbf{t}}_{\alpha}\overset{\eqref{eq:big-t-bar},\eqref{eq:partial-subdiff}}{=}\underset{i\in V_{4}\backslash\{\alpha\}}{\overset{\phantom{i\in V_{4}\backslash\{\alpha\}}}{\sum}}[\mathbf{z}_{i}^{n,\bar{w}}-\hat{\mathbf{z}}_{i}^{n,\bar{w}}]\end{array}\label{eq:two-sum-bdds-a}\\
 & \mbox{ and } & \begin{array}{c}
\|\mathbf{v}_{A}^{n,\bar{w}}-\hat{\mathbf{v}}_{A}^{n,\bar{w}}\|\overset{\eqref{eq_m:all_acronyms},\eqref{eq:v-hat}}{\leq}\underset{i\in V_{4}}{\overset{\phantom{i\in V_{4}}}{\sum}}\|\mathbf{z}_{i}^{n,\bar{w}}-\hat{\mathbf{z}}_{i}^{n,\bar{w}}\|.\end{array}\label{eq:two-sum-bdds-b}
\end{eqnarray}
\end{subequations}Moreover, 
\begin{equation}
\begin{array}{c}
\underset{i\in V_{4}\backslash\{\alpha\}}{\overset{\phantom{i\in V_{4}\backslash\{\alpha\}}}{\sum}}\|\mathbf{z}_{i}^{n,\bar{w}}-\hat{\mathbf{z}}_{i}^{n,\bar{w}}\|\overset{\eqref{eq:def-delta-z-i}}{\leq}\underset{i\in V_{4}}{\overset{\phantom{i\in V_{4}}}{\sum}}\|\Delta z_{i}\|\overset{\eqref{eq:gradient-sqrt-bdd}}{\leq}\underset{i\in V_{4}}{\overset{\phantom{i\in V_{4}}}{\sum}}\sqrt{\Delta f_{i}}.\end{array}\label{eq:v-hat-leashed}
\end{equation}

If $\Delta f_{i}$ were arbitrarily large, then Lemma \ref{lem:subdiff-decrease}(1)
would contradict the fact that the dual objective value is monotonically
nonincreasing. Lemma \ref{lem:grad-sqrt-bdd} then shows that $\Delta z_{i}$
is bounded. 

Next, Theorem \ref{thm:convergence}(iii) shows that $\mathbf{z}_{i}^{n,p(n,i)}$
is bounded for all $i\in V_{3}\cup V_{4}$. Along with the fact that
$\Delta z_{i}$ is bounded for all $i\in V_{4}$, we have $\hat{\mathbf{z}}_{i}^{n,p(n,i)}$
being bounded for all $i\in V_{4}$. 

Let $h^{n,w}=-F^{n,0}(\mathbf{z}^{n,0})-(-F(\mathbf{z}^{*}))$, and
let $h^{n}$ be defined by 
\begin{equation}
h^{n}:=h^{n,0}=-F^{n,0}(\mathbf{z}^{n,0})-(-F(\mathbf{z}^{*})).\label{eq:def-h-n-values}
\end{equation}
Note that $h^{n}\geq0$. From the fact that $\mathbf{z}_{\alpha}^{n,p(n,\alpha)}\overset{\eqref{eq:stagnant-indices}}{=}\mathbf{z}_{\alpha}^{n,\bar{w}}$,
we have 
\begin{eqnarray}
\|\mathbf{t}_{\alpha}-\bar{\mathbf{t}}_{\alpha}\| & \overset{\eqref{eq:partial-subdiff},\eqref{eq:big-t-bar},\eqref{eq:stagnant-indices}}{=} & \begin{array}{c}
\bigg\|\underset{\beta\in V\cup\bar{E}}{\overset{}{\sum}}[\mathbf{z}_{\beta}^{n,p(n,\alpha)}-\mathbf{z}_{\beta}^{n,\bar{w}}]\bigg\|\end{array}\label{eq:t-norm-bound}\\
 & \leq & \begin{array}{c}
\underset{w=p(n,\alpha)+1}{\overset{\bar{w}}{\sum}}\bigg\|\underset{\beta\in V\cup\bar{E}}{\overset{}{\sum}}[\mathbf{z}_{\beta}^{n,w-1}-\mathbf{z}_{\beta}^{n,w}]\bigg\|\end{array}\nonumber \\
 & \overset{\eqref{eq_m:all_acronyms}}{\leq} & \begin{array}{c}
\underset{w=1}{\overset{\bar{w}}{\sum}}\|\mathbf{v}_{A}^{n,w-1}-\mathbf{v}_{A}^{n,w}\|\leq\sqrt{\bar{w}\underset{w=1}{\overset{\bar{w}}{\sum}}\|\mathbf{v}_{A}^{n,w-1}-\mathbf{v}_{A}^{n,w}\|^{2}}\end{array}\nonumber \\
 & \overset{\scriptsize{\mbox{Thm \ref{thm:convergence}(i)}}}{\leq} & \begin{array}{c}
\sqrt{\bar{w}}\sqrt{h^{n}-h^{n+1}}.\end{array}\nonumber 
\end{eqnarray}

\begin{thm}
\label{thm:main-conv-rate}Suppose that a dual optimizer $\mathbf{z}^{*}$
exists for \eqref{eq:general-dual}. Suppose that the dual iterates
$\{\mathbf{z}_{\alpha}^{n,w}\}$ are bounded for all $\alpha\in V\cup\bar{E}$,
$w\in\{1,\dots,\bar{w}\}$ and $n\geq1$. Then we have the recurrence
\[
h^{n+2}\leq h^{n}-\gamma[h^{n+2}]^{4}.
\]
for some $\gamma>0$, which together with Lemma \ref{lem:Beck-recur}
shows that the dual function value converges at a rate of $O(1/n^{1/3})$.
Moreover, if $V_{4}=\emptyset$, we have the recurrence $h^{n+1}\leq C_{1}\sqrt{\bar{w}}\sqrt{h^{n}-h^{n+1}}$
for some $C_{1}>0$, which shows that the dual function value converges
at a rate of $O(1/n)$. 
\end{thm}

\begin{proof}
We make the following claim:

\textbf{Claim: }Throughout Algorithm \ref{alg:Ext-Dyk}, the quantities
$\mathbf{v}_{A}^{n,w}$, $\mathbf{x}^{n,w}$ and $\hat{\mathbf{v}}_{A}^{n,w}$
are bounded. 

We recall that Theorem \ref{thm:convergence}(ii) implies that $\{\mathbf{v}_{A}^{n,w}\}$
is bounded. Since $\mathbf{x}^{n,w}\overset{\eqref{eq:x-from-v-A}}{=}\bar{\mathbf{x}}-\mathbf{v}_{A}^{n,w}$,
$\{\mathbf{x}^{n,w}\}$ is also bounded. From Lemma \ref{lem:subdiff-decrease}(1)
and the nonincreasingness of $\{-F(\mathbf{z}^{n,w})\}$ through Theorem
\ref{thm:convergence}(i), we deduce that $\Delta f_{i}$ is bounded.
From \eqref{eq:v-hat-leashed} and \eqref{eq:two-sum-bdds-b}, we
can deduce that $\{\hat{\mathbf{v}}_{A}^{n,w}\}$ is bounded.  $\hfill\triangle$

Note that $\{\mathbf{z}_{\alpha}^{n,w}\}$ is assumed to be bounded
for all $\alpha\in V\cup\bar{E}$ in the theorem statement. Combining
with the claim above with \eqref{eq_m:two-sum-bdds}, \eqref{eq:v-hat-leashed}
and \eqref{eq:t-norm-bound} shows us that there are nonnegative constants
$C_{1}$, $C_{2}$ and $C_{3}$ such that 
\begin{eqnarray}
h^{n+1} & \overset{\eqref{eq:def-h-n-values}}{=} & -[F^{n,\bar{w}}(\mathbf{z}^{n,\bar{w}})-F(\hat{\mathbf{z}}^{n,\bar{w}})]-[F(\hat{\mathbf{z}}^{n,\bar{w}})-F(\mathbf{z}^{*})]\label{eq:recurrence}\\
 & \overset{\eqref{eq:subdiff-for-h},\eqref{eq:subdiff-part-2}}{\leq} & \begin{array}{c}
C_{1}\sqrt{\bar{w}}\sqrt{h^{n}-h^{n+1}}+\underset{i\in V_{4}}{\overset{\phantom{i\in V_{4}}}{\sum}}[C_{2}\Delta f_{i}+C_{3}\sqrt{\Delta f_{i}}].\end{array}\nonumber 
\end{eqnarray}
 We now address the last statement of the theorem first. When $V_{4}=\emptyset$,
the formula \eqref{eq:recurrence} is reduced to $h^{n+1}\leq C_{1}\sqrt{\bar{w}}\sqrt{h^{n}-h^{n+1}}$,
which shows that $h^{n}\geq h^{n+1}+\frac{1}{C_{1}^{2}\bar{w}}[h^{n+1}]^{2}$.
This recurrence would give us an $O(1/n)$ convergence rate from \cite[Lemma 3.5]{Beck_Tetruashvili_2013}. 

We now continue to proving our main result. Let $L=\max_{i\in V_{4}}L_{i}$.
From Lemma \ref{lem:subdiff-decrease}(1) being applied to all coordinates
$i$ in $V_{4}$, we have the constraints 
\begin{eqnarray}
 &  & \begin{array}{c}
h^{n+1,1}\leq h^{n+1,0}-\frac{1}{2}\underset{i\in V_{4}}{\overset{\phantom{i\in V_{4}}}{\sum}}\theta_{i}^{2},\end{array}\label{eq:recur-1}\\
 & \mbox{and} & \begin{array}{c}
\frac{1}{2(L+1)^{2}}\theta_{i}^{2}+\theta_{i}\geq f_{i}([\mathbf{x}^{n+1,0}]_{i})-f_{i,n,\bar{w}}([\mathbf{x}^{n+1,0}]_{i})\mbox{ for all }i\in V_{4}.\end{array}\label{eq:recur-2}
\end{eqnarray}
Due to Assumption \ref{assu:to-start-subalg}(2) and \eqref{eq:def-delta-f-i},
the right hand side of \eqref{eq:recur-2} is $\Delta f_{i}$. From
the fact that $\{h^{n,w}\}$ is nonincreasing and bounded from below,
there is some constant $C_{4}$ such that $\theta_{i}\leq C_{4}$.
So 
\begin{equation}
\begin{array}{c}
\Delta f_{i}\overset{\eqref{eq:recur-2}}{\leq}\theta_{i}+\frac{1}{2(L+1)^{2}}\theta_{i}^{2}\leq\theta_{i}\left(1+\frac{C_{4}}{2(L+1)^{2}}\right)\leq C_{4}\left(1+\frac{C_{4}}{2(L+1)^{2}}\right).\end{array}\label{eq:recur-3}
\end{equation}
Letting $C_{5}$ be $(1+\frac{C_{4}}{2(L+1)^{2}})$ and $C_{6}=C_{4}C_{5}$,
we have 
\begin{equation}
\begin{array}{c}
C_{2}\Delta f_{i}+C_{3}\sqrt{\Delta f_{i}}\leq\sqrt{\Delta f_{i}}\left(C_{2}\sqrt{C_{6}}+C_{3}\right).\end{array}\label{eq:recur-4}
\end{equation}
Also, 
\begin{equation}
\begin{array}{c}
h^{n+2}\overset{\eqref{eq:recur-1}}{\leq}h^{n+1}-\frac{1}{2}\underset{i\in V_{4}}{\overset{\phantom{i\in V_{4}}}{\sum}}\theta_{i}^{2}\overset{\eqref{eq:recur-3}}{\leq}h^{n+1}-\frac{1}{2C_{5}^{2}}\underset{i\in V_{4}}{\overset{\phantom{i\in V_{4}}}{\sum}}[\Delta f_{i}]^{2}\leq h^{n+1}-\frac{1}{2C_{5}^{2}|V_{4}|}\Big[\underset{i\in V_{4}}{\overset{\phantom{i\in V_{4}}}{\sum}}\Delta f_{i}\Big]^{2}.\end{array}\label{eq:recur-5}
\end{equation}
Let $C_{7}=C_{2}\sqrt{C_{6}}+C_{3}$. We have, 
\begin{eqnarray}
h^{n+1} & \overset{\eqref{eq:recurrence},\eqref{eq:recur-4}}{\leq} & \begin{array}{c}
C_{1}\sqrt{\bar{w}}\sqrt{h^{n}-h^{n+1}}+\underset{i\in V_{4}}{\overset{\phantom{i\in V_{4}}}{\sum}}C_{7}\sqrt{\Delta f_{i}}\end{array}\label{eq:recur-6}\\
 & \leq & \begin{array}{c}
C_{1}\sqrt{\bar{w}}\sqrt{h^{n}-h^{n+1}}+C_{7}\sqrt{|V_{4}|}\sqrt{\sum_{i\in V_{4}}\Delta f_{i}}.\end{array}\nonumber 
\end{eqnarray}
We now split into two cases to find a recurrence. 

\textbf{Case 1:} If $h^{n+1}\leq2C_{1}\sqrt{\bar{w}}\sqrt{h^{n}-h^{n+1}}$,
then $h^{n}\geq h^{n+1}+\frac{1}{4C_{1}^{2}\bar{w}}[h^{n+1}]^{2}$.
There is some $\bar{h}$ such that $h^{n}\leq\bar{h}$ for all $n$,
which gives 
\begin{equation}
\begin{array}{c}
h^{n}\geq h^{n+1}+\frac{1}{4C_{1}^{2}\bar{w}}[h^{n+1}]^{2}\ge h^{n+2}+\frac{1}{4C_{1}^{2}\bar{h}_{1}^{2}\bar{w}}[h^{n+2}]^{4}.\end{array}\label{eq:recur-8-0}
\end{equation}

\textbf{Case 2:} If $h^{n+1}\geq2C_{1}\sqrt{\bar{w}}\sqrt{h^{n}-h^{n+1}}$,
then 
\begin{equation}
\begin{array}{c}
\sqrt{\underset{i\in V_{4}}{\sum}\Delta f_{i}}\overset{\eqref{eq:recur-6}}{\geq}\frac{1}{C_{7}\sqrt{|V_{4}|}}[h^{n+1}-C_{1}\sqrt{\bar{w}}\sqrt{h^{n}-h^{n+1}}]\geq\frac{1}{2C_{7}\sqrt{|V_{4}|}}h^{n+1}.\end{array}\label{eq:recur-7}
\end{equation}
This gives 
\begin{equation}
\begin{array}{c}
h^{n+2}\overset{\eqref{eq:recur-5},\eqref{eq:recur-7}}{\leq}h^{n+1}-\frac{1}{32C_{5}^{2}|V_{4}|^{3}C_{7}^{4}}[h^{n+1}]^{4}\leq h^{n}-\frac{1}{32C_{5}^{2}|V_{4}|^{3}C_{7}^{4}}[h^{n+2}]^{4}.\end{array}\label{eq:recur-8}
\end{equation}
The recurrences \eqref{eq:recur-8-0} and \eqref{eq:recur-8} ensure
that the conclusion holds. 
\end{proof}
The following result is adapted from the techniques in \cite{Beck_Tetruashvili_2013,Beck_alt_min_SIOPT_2015}.
\begin{lem}
\label{lem:Beck-recur}Suppose that a nonnegative sequence $\{a_{k}\}_{k=1}^{\infty}$
has the recurrence $a_{k}\geq a_{k+1}+\gamma a_{k+1}^{4}$. Then $a_{k}\leq\left(\frac{1}{a_{1}^{3}}+(k-1)3\gamma\left(3\gamma a_{1}^{3}+1\right)^{-1}\right)^{-1/3}$
for all $k\geq1$, which means that $\{a_{k}\}_{k}$ has a $O(1/k^{1/3})$
rate of convergence. 
\end{lem}

\begin{proof}
We have  
\begin{eqnarray*}
\begin{array}{c}
\frac{1}{a_{k+1}^{3}}-\frac{1}{a_{k}^{3}}\end{array} & = & \begin{array}{c}
\frac{a_{k+1}^{3}-a_{k}^{3}}{a_{k+1}^{3}a_{k}^{3}}=\frac{(a_{k+1}-a_{k})(a_{k+1}^{2}+a_{k+1}a_{k}+a_{k}^{2})}{a_{k+1}^{3}a_{k}^{3}}\end{array}\\
 & \geq & \begin{array}{c}
\frac{\gamma a_{k+1}^{4}3a_{k+1}^{2}}{a_{k+1}^{3}a_{k}^{3}}\geq3\gamma\frac{a_{k+1}^{3}}{a_{k}^{3}}.\end{array}
\end{eqnarray*}
Another bound is 
\[
\begin{array}{c}
\frac{1}{a_{k+1}^{3}}-\frac{1}{a_{k}^{3}}=\frac{1}{a_{k+1}^{3}}\left(1-\frac{a_{k+1}^{3}}{a_{k}^{3}}\right)\geq\frac{1}{a_{1}^{3}}\left(1-\frac{a_{k+1}^{3}}{a_{k}^{3}}\right).\end{array}
\]
It is elementary to calculate that $\max\left\{ 3\gamma\frac{a_{k+1}^{3}}{a_{k}^{3}},\frac{1}{a_{1}^{3}}\left(1-\frac{a_{k+1}^{3}}{a_{k}^{3}}\right)\right\} $
has a minimum value of $\left(3\gamma a_{1}^{3}+1\right)^{-1}3\gamma$
attained at $\frac{a_{k+1}^{3}}{a_{k}^{3}}=\left(3\gamma a_{1}^{3}+1\right)^{-1}$.
Then 
\[
\begin{array}{c}
\frac{1}{a_{k}^{3}}\geq\frac{1}{a_{1}^{3}}+(k-1)3\gamma\left(3\gamma a_{1}^{3}+1\right)^{-1},\end{array}
\]
which gives us the required conclusion.
\end{proof}
As a corollary of Theorem \ref{thm:CQ-gives-bdd-iters}, we have the
following.
\begin{cor}
Suppose that the condition \eqref{eq:CQ} holds. Then the iterates
$\{\mathbf{z}_{i}^{n,w}\}_{i\in V}$ of Algorithm \ref{alg:Ext-Dyk}
are bounded, which implies that Theorem \ref{thm:main-conv-rate}
can be applied to show the $O(1/n^{1/3})$ convergence rate of Algorithm
\ref{alg:Ext-Dyk}. 
\end{cor}

\section{Numerical experiments}

We present our numerical experiments. Since a distributed optimization
algorithm is designed to handle the distributed nature of the data
and keeping the communications between the nodes low, a distributed
algorithm would converge less quickly than a comparable centralized
algorithm. So we aim only to verify the theoretical rates obtained
in this paper.

Since the distributed Dykstra's algorithm extends the averaged consensus
algorithm, the kind of graph that the distributed Dykstra's algorithm
does best in is one where the degree of each node is relatively high
so that each node can actively seek neighbors to average their primal
variable with (which occurs when $S_{n,w}$ is the edge connecting
the two nodes). Nevertheless, we are keeping our experiments simple
by looking at the graph where $|V|=5$ and $E=\{\{1,2\},\{1,3\},\{1,4\},\{1,5\}\}$.
We look at the setting of Example \ref{exa:distrib-dyk} where $X_{i}=\mathbb{R}^{m}$
and $m=4$ for all $i\in V$, and look at halfspaces of the form 
\[
H_{(i,j)}=\{\mathbf{{x}}\in\mathbf{{X}}:[\mathbf{x}]_{i}=[\mathbf{\mathbf{x}}]_{j}\}
\]
instead of the halfspaces $H_{((i,j),k)}$ defined in \eqref{eq:H-alpha-subspaces}
to simplify computations. Let $\mathbf{e}$ be \texttt{ones(m,1)}.
First, we find $\{v_{i}\}_{i\in V}$ and $\bar{x}$ such that $\sum_{i\in V}v_{i}+|V|(\mathbf{e}-\bar{x})=0$.
We then find closed convex functions $f_{i}(\cdot)$ such that $v_{i}\in\partial f_{i}(\mathbf{e})$.
It is clear from the KKT conditions that $\mathbf{e}$ is the primal
optimum solution to \eqref{eq:distrib-dyk-primal-pblm} if $[\bar{\mathbf{x}}]_{i}$
are all equal to $\bar{x}$ for all $i\in V$. 

The $f_{i}(\cdot)$ can be defined as either smooth or nonsmooth functions,
or as the indicator functions of level sets of smooth or nonsmooth
functions. They are described using some Matlab functions below. 
\begin{itemize}
\item [(F-S)]$f_{i}(x):=\frac{1}{2}x^{T}A_{i}x+b_{i}^{T}x+c_{i}$, where
$A_{i}$ is of the form $vv^{T}+rI$, where $v$ is generated by \texttt{rand(m,1)},
$r$ is generated by \texttt{rand(1)}. $b_{i}$ is chosen to be such
that $v_{i}=\nabla f(\mathbf{e})$, and $c_{i}=0$. 
\item [(F-NS)]$f_{i}(x):=\max\{f_{i,1}(x),f_{i,2}(x)\}$, where $f_{i,j}(x):=\frac{1}{2}x^{T}A_{i}x+b_{i,j}^{T}x+c_{i,j}$
for $j\in\{1,2\}$, $A_{i}$ is of the form $vv^{T}+rI$, where $v$
is generated by \texttt{rand(m,1)}, $r$ is generated by \texttt{rand(1)},
$b_{i,1}$ and $b_{i,2}$ are chosen such that $v_{i}=\frac{1}{2}[\nabla f_{i,1}(\mathbf{e})+\nabla f_{i,2}(\mathbf{e})]$
but $v_{i}$ is neither $\nabla f_{i,1}(\mathbf{e})$ nor $\nabla f_{i,2}(\mathbf{e})$,
and $c_{i,1}$ and $c_{i,2}$ are chosen such that $f_{i,1}(\mathbf{e})=f_{i,2}(\mathbf{e})$. 
\end{itemize}
Our code is equivalent to $\bar{w}=8$ with 
\begin{align*}
 & S_{n,1}=\{(1,2)\},S_{n,2}=\{1,2\},S_{n,3}=\{(1,3)\},S_{n,4}=\{1,3\},\\
 & S_{n,5}=\{(1,4)\},S_{n,6}=\{1,4\},S_{n,7}=\{(1,5)\},\text{ and }S_{n,8}=\{1,5\}.
\end{align*}

From the analysis in \cite{Pang_Dist_Dyk,Pang_sub_Dyk} (which traces
its origins to \cite{Gaffke_Mathar}), the duality gap is bounded
from below by 
\begin{equation}
0\leq\frac{1}{2}\|\mathbf{x}^{n,w}-\mathbf{x}^{*}\|^{2}\leq\underbrace{\frac{1}{2}\|\mathbf{x}^{*}-\bar{\mathbf{x}}\|^{2}+\sum_{i\in V}\mathbf{f}_{i}(\mathbf{x})-F^{n,w}(\{\mathbf{z}_{\alpha}\}_{\alpha\in\bar{E}\cup V})}_{\scriptsize{\text{duality gap modeling \eqref{eq:def-h-n-values}}}}.\label{eq:num_exp_ineq}
\end{equation}
We will keep track of the values of $\frac{1}{2}\|\mathbf{x}^{n,w}-\mathbf{x}^{*}\|^{2}$
and the duality gap as marked. Note that the duality gap is monotonically
nonincreasing. 

We now report on the results of the numerical experiments, starting
with the case of smooth functions and see the effect of treating the
smooth functions $f_{i}(\cdot)$ as subdifferentiable functions (i.e.,
being in $V_{4}$) and as proximable functions (i.e., being in $V_{1}$).
The theory in our paper suggests linear convergence, which was observed.
One might expect that if we treat the $f_{i}(\cdot)$ as proximable
functions, the dual objective value in \eqref{eq:h-def} or its lower
estimate \eqref{eq:h-def-1} converges to the optimal value faster.
While this is mostly true, we have encountered settings where treating
the smooth functions as subdifferentiable can give faster decrease
in the dual objective value. We ran our experiments 40000 times, and
in 5002 times, treating the smooth function as a subdifferentiable
function results in a lower duality gap by the $200$th iteration.
This is illustrated in the first diagram in Figure \ref{fig:the-fig}. 

We now look at the nonsmooth case. For the case when we treat the
functions as subdifferentiable functions, a plot of the duality gap
over the iterations shows that the convergence rate of the duality
gap to zero is $O(1/n)$, which coincides with the theory. (See Figure
\ref{fig:the-fig}, 4th diagram.) We make two observations that cannot
be predicted by our theory so far. The first observation we see is
that $\frac{1}{2}\|\mathbf{x}^{n,w}-\mathbf{x}^{*}\|^{2}$ apparently
converges to zero at the rate of $O(1/n^{2})$. (Equivalently, $\|\mathbf{x}^{n,w}-\mathbf{x}^{*}\|$
converges to zero at the rate of $O(1/n)$. See Figure \ref{fig:the-fig},
5th diagram.) The next observation is that when we treat the functions
as proximable, the duality gap and the distance to the optimal solution
converges linearly to zero. We ran more than 300 experiments, and
found that this linear convergence always holds. (See Figure \ref{fig:the-fig},
2nd diagram.)

\begin{figure}
\begin{raggedright}
\includegraphics[scale=0.3]{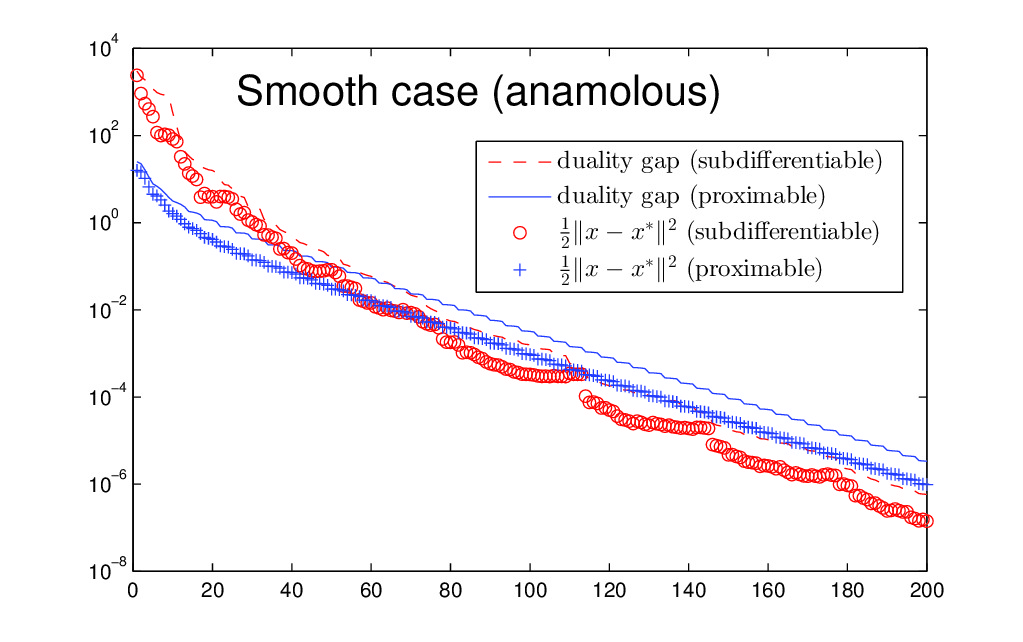}$\!\!\!\!\!\!\!$\includegraphics[scale=0.3]{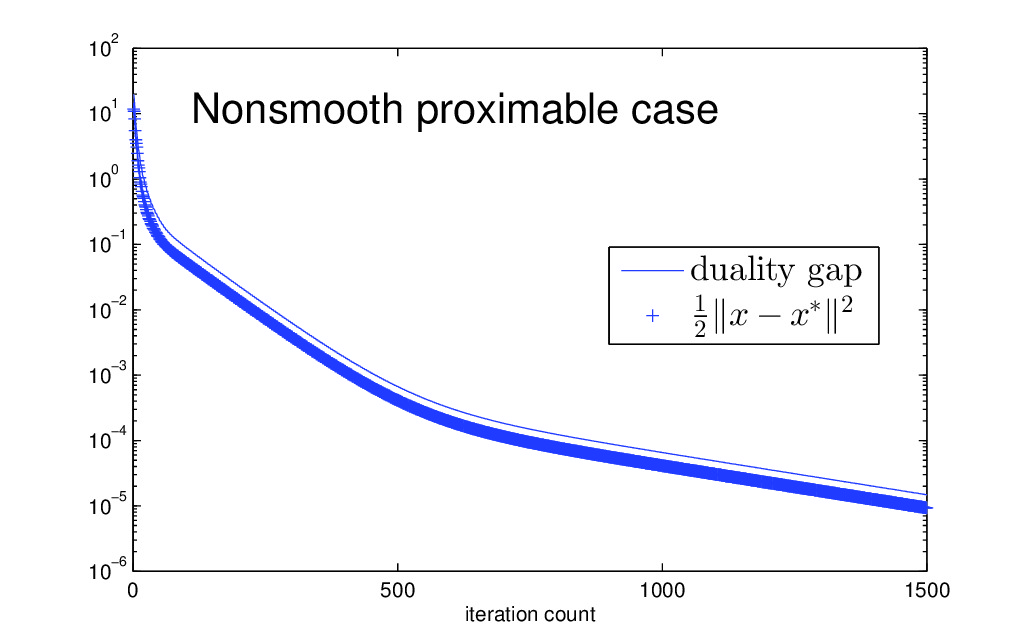}$\!\!\!\!\!\!\!$\includegraphics[scale=0.3]{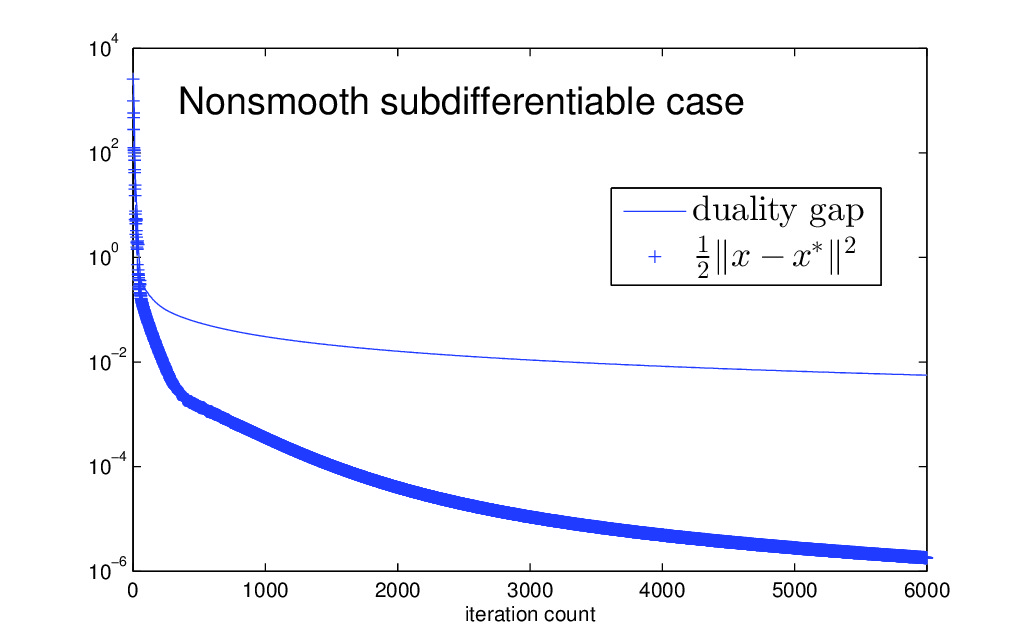}
\par\end{raggedright}
\begin{raggedright}
\includegraphics[scale=0.3]{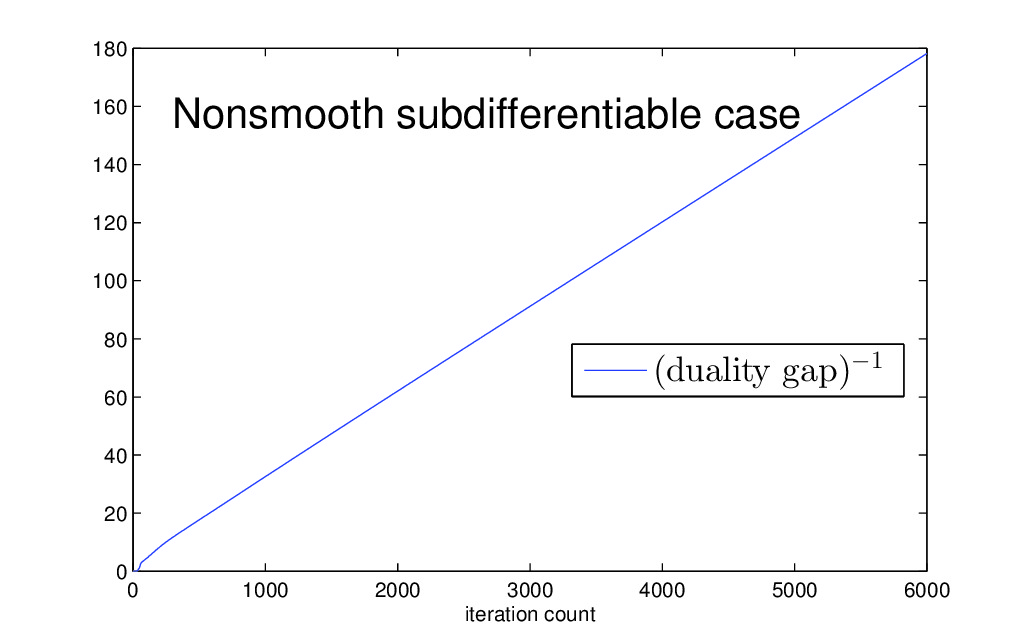}$\!\!\!\!\!\!\!$\includegraphics[scale=0.3]{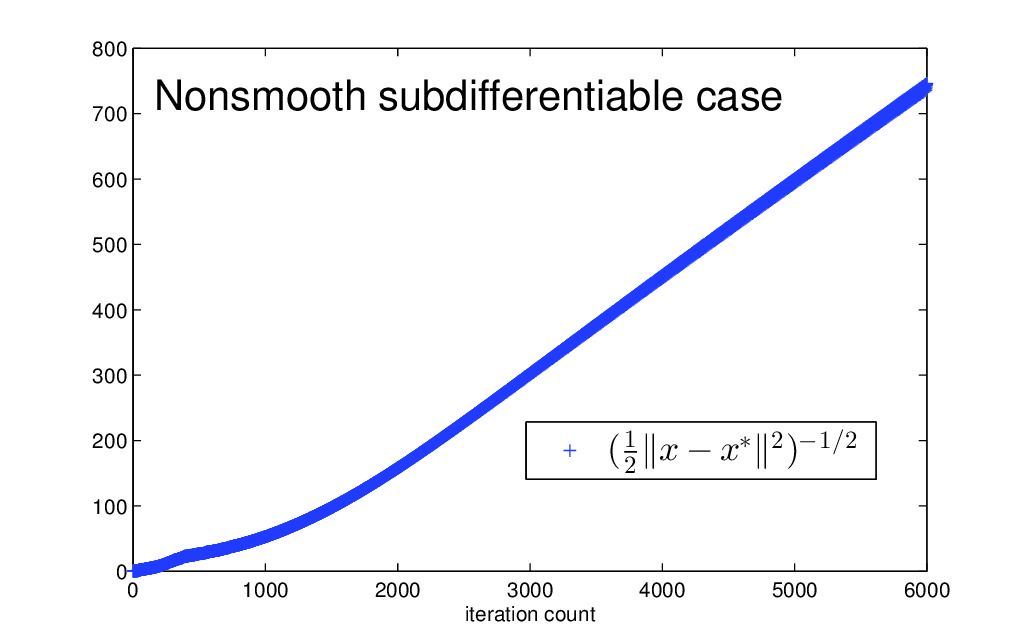}$\!\!\!\!\!\!\!$\includegraphics[scale=0.3]{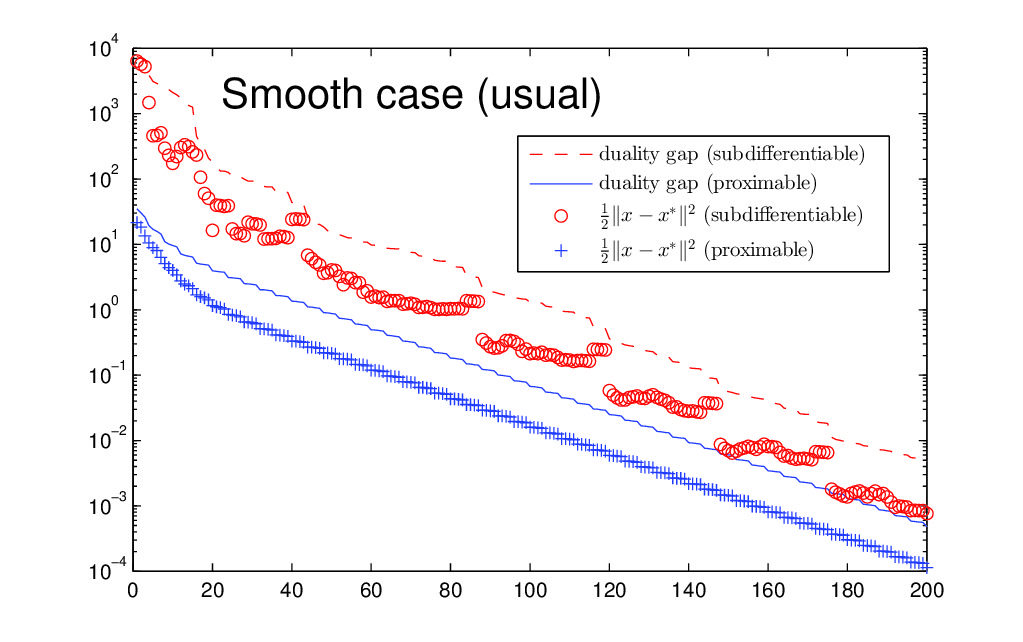}
\par\end{raggedright}
\caption{\label{fig:the-fig}This figure illustrates a sample run of our numerical
experiments. The first diagram shows the rate of linear convergence
for the smooth case when all (smooth) functions are treated as subdifferentiable
functions and when all functions are treated as proximable functions.
In this case, we see the anomalous case of when the algorithm actually
runs faster when the functions are treated as subdifferentiable. (The
last diagram shows what usually happens.) The second and third diagrams
are for the nonsmooth case, and illustrate the linear convergence
rate (at a much slower rate than the smooth case) when all functions
are treated as proximable, and the sublinear convergence rate for
when all functions are treated as subdifferentiable. The diagrams
suggest that the duality gap converges at a rate of $O(1/n)$ (as
is suggested by the theory) and $\frac{1}{2}\|\mathbf{x}-\mathbf{x}^{*}\|^{2}$
converges at a rate of $O(1/n^{2})$ (which is not covered by the
theory). }
\end{figure}

\section{Conclusion}

We proved what we have set out to do in Subsection \ref{subsec:Contributions}.
The linear convergence and $O(1/k)$ rates in Sections \ref{sec:Lin-conv}
and \ref{sec:O-1-k-conv-of-smooth-proximable-case} cannot be improved
to a faster rate, but it is unclear whether the $O(1/k^{1/3})$ rate
in Section \ref{sec:Sublin-conv} is optimal. Indeed, our numerical
experiments suggest a rate of $O(1/k)$, and there might be reasonable
conditions leading to the linear convergence observed for the nonsmooth
proximable case. These require further investigation.

\bibliographystyle{amsalpha}
\bibliography{../refs}

\newcommand{\etalchar}[1]{$^{#1}$}
\providecommand{\bysame}{\leavevmode\hbox to3em{\hrulefill}\thinspace}
\providecommand{\MR}{\relax\ifhmode\unskip\space\fi MR }
% \MRhref is called by the amsart/book/proc definition of \MR.
\providecommand{\MRhref}[2]{%
  \href{http://www.ams.org/mathscinet-getitem?mr=#1}{#2}
}
\providecommand{\href}[2]{#2}
\begin{thebibliography}{VHDG11}

\bibitem[AFJ16]{Aytekin_F_Johansson_2016}
A.~Aytekin, H.R. Feyzmahdavian, and M.~Johansson, \emph{Analysis and
  implementation of an asynchronous optimization algorithm for the parameter
  server}, arxiv eprint 1610.05507, 2016.

\bibitem[AG17]{Aybat_Gurb_effective_resistance_2017}
N.S. Aybat and M.~Gurbuzbalaban, \emph{Decentralized computation of effective
  resistances and acceleration of consensus algorithms}, arxiv eprints:
  https://arxiv.org/abs/1708.07190, 2017.

\bibitem[AH16]{Aybat_Hamedani_2016}
N.S. Aybat and E.Y. Hamedani, \emph{A primal-dual method for conic constrained
  distributed optimization problems}, Advances in Neural Information Processing
  Systems 29, Curran associates, Red Hook, NY, 2016, pp.~5049--5057.

\bibitem[BB96]{BB96_survey}
H.H. Bauschke and J.M. Borwein, \emph{On projection algorithms for solving
  convex feasibility problems}, SIAM Rev. \textbf{38} (1996), 367--426.

\bibitem[BBL99]{BBL99}
H.H. Bauschke, J.M. Borwein, and W.~Li, \emph{Strong conical hull intersection
  property, bounded linear regularity, {J}ameson's property ({G}), and error
  bounds in convex optimization}, Math. Program., Ser. A \textbf{86} (1999),
  no.~1, 135--160.

\bibitem[BC11]{BauschkeCombettes11}
H.H. Bauschke and P.L. Combettes, \emph{Convex analysis and monotone operator
  theory in {H}ilbert spaces}, Springer, 2011.

\bibitem[BCN{\etalchar{+}}17]{Notarstefano_gang_Newton_2017}
N.~Bof, R.~Carli, G.~Notarstefano, L.~Schenato, and D.~Varagnolo,
  \emph{{N}ewton-{R}aphson consensus under asynchronous and lossy
  communications for peer-to-peer networks}, 2017.

\bibitem[BCS17]{Bof_Carli_Schenato_2017}
N.~Bof, R.~Carli, and L.~Schenato, \emph{Average consensus with asynchronous
  updates and unreliable communication}, Proc. of the IFAC Word Congress, 2017,
  pp.~601--606.

\bibitem[BD85]{BD86}
J.P. Boyle and R.L. Dykstra, \emph{A method for finding projections onto the
  intersection of convex sets in {H}ilbert spaces}, Advances in Order
  Restricted Statistical Inference, Lecture notes in Statistics, Springer, New
  York, 1985, pp.~28--47.

\bibitem[Bec15]{Beck_alt_min_SIOPT_2015}
A.~Beck, \emph{On the convergence of alternating minimization for convex
  programming with applications to iteratively reweighted least squares and
  decomposition schemes}, SIAM J. Optim. \textbf{25} (2015), no.~1, 185--209.

\bibitem[BHI14]{Bianchi_Hachem_Iutzeler_2nd_paper_dist}
P.~Bianchi, W.~Hachem, and F.~Iutzeler, \emph{A stochastic coordinate descent
  primal-dual algorithm and applications}, 2014 IEEE International workshop on
  machine learning for signal processing, Sept. 21--24 2014.

\bibitem[BPC{\etalchar{+}}10]{Boyd_Eckstein_ADMM_review}
S.~Boyd, N.~Parikh, E.~Chu, B.~Peleato, and J.~Eckstein, \emph{Distributed
  optimization and statistical learning via the alternating direction method of
  multipliers}, Foundations and Trends in Machine Learning \textbf{3} (2010),
  no.~1, 1--122.

\bibitem[BT13]{Beck_Tetruashvili_2013}
A.~Beck and L.~Tetruashvili, \emph{On the convergence of block coordinate
  descent type methods}, SIAM J. Optim. \textbf{23} (2013), no.~4, 2037--2060.

\bibitem[CE18]{Eckstein_Combettes_MAPR}
P.L. Combettes and J.~Eckstein, \emph{Asynchronous block-iterative primal-dual
  decomposition methods for monotone inclusions}, Math. Prog. Ser. B
  \textbf{168} (2018), no.~1-2, 645--672.

\bibitem[CP11]{Chambolle_Pock_2011_PD}
A.~Chambolle and T.~Pock, \emph{A first-order primal-dual algorithm for convex
  problems with applications to imaging}, J. Math. Imaging Vis. \textbf{40}
  (2011), no.~1, 120--145.

\bibitem[CP16]{Chambolle_Pock_MAPR}
\bysame, \emph{On the ergodic convergence rates of a first-order primal-dual
  algorithm}, Math, Prog. \textbf{159} (2016), 253--287.

\bibitem[Deu01a]{Deutsch01_survey}
F.~Deutsch, \emph{Accelerating the convergence of the method of alternating
  projections via a line search: A brief survey}, Inherently Parallel
  Algorithms in Feasibility and Optimization and their Applications
  (D.~Butnariu, Y.~Censor, and S.~Reich, eds.), Elsevier, 2001, pp.~203--217.

\bibitem[Deu01b]{Deustch01}
\bysame, \emph{Best approximation in inner product spaces}, Springer, 2001, CMS
  Books in Mathematics.

\bibitem[Dyk83]{Dykstra83}
R.L. Dykstra, \emph{An algorithm for restricted least-squares regression}, J.
  Amer. Statist. Assoc. \textbf{78} (1983), 837--842.

\bibitem[ER11]{EsRa11}
R.~Escalante and M.~Raydan, \emph{Alternating projection methods}, SIAM, 2011.

\bibitem[GM89]{Gaffke_Mathar}
N.~Gaffke and R.~Mathar, \emph{A cyclic projection algorithm via duality},
  Metrika \textbf{36} (1989), 29--54.

\bibitem[GOP17]{Gurbuzbalaban_Ozdaglar_Parrilo_SIOPT_2017}
M.~Gurbuzbalaban, A.~Ozdaglar, and P.~Parrilo, \emph{On the convergence rate of
  incremental aggregated gradient algorithms}, SIAM J. Optim. \textbf{27}
  (2017), no.~2, 1035--1048.

\bibitem[Han88]{Han88}
S.P. Han, \emph{A successive projection method}, Math. Programming \textbf{40}
  (1988), 1--14.

\bibitem[HD97]{Hundal-Deutsch-97}
H.S. Hundal and F.~Deutsch, \emph{Two generalizations of {D}ykstra's cyclic
  projections algorithm}, Math. Programming \textbf{77} (1997), 335--355.

\bibitem[IBCH13]{Iutzeler_Bianchi_Ciblat_Hachem_1st_paper_dist}
F.~Iutzeler, P.~Bianchi, P.~Ciblat, and W.~Hachem, \emph{Asynchronous
  distributed optimization using a randomized alternating direction method of
  multipliers}, Proceedings of the 52nd Conference on Decision Control
  (Florence, Italy), Dec. 2013, pp.~3671--3676.

\bibitem[Kru06]{Kruger_06}
A.Y. Kruger, \emph{About regularity of collections of sets}, Set-Valued Anal.
  \textbf{14} (2006), 187--206.

\bibitem[Ned15]{Nedich_survey}
A.~Nedich, \emph{Convergence rate of distributed averaging dynamics and
  optimization in networks}, Foundations and Trends in Systems and Control
  \textbf{2} (2015), no.~1, 1--100.

\bibitem[Ned17]{Nedich_talk_2017}
\bysame, \emph{Fast algorithms for distributed optimization over time-varying
  graphs}, Talk at DIMACS Workshop on Distributed Optimization, Information
  Processing, and Learning,
  https://www.nrel.gov/grid/assets/pdfs/aeg-nedich.pdf, 2017.

\bibitem[NO15]{Nedich_Olshevsky}
A.~Nedich and A.~Olshevsky, \emph{Distributed optimization over time-varying
  directed graphs}, IEEE Transactions on Automatic Control \textbf{60} (2015),
  no.~3, 601--615.

\bibitem[NOS17]{Nedich_Olshevsky_Shi}
A.~Nedich, A.~Olshevsky, and W.~Shi, \emph{Achieving geometric convergence for
  distributed optimization over time-varying graphs}, SIAM J. Optim.
  \textbf{27} (2017), no.~4, 2597--2633.

\bibitem[Pan18a]{Pang_Dist_Dyk}
C.H.J. Pang, \emph{Distributed deterministic asynchronous algorithms in
  time-varying graphs through {D}ykstra splitting}, 2018.

\bibitem[Pan18b]{Pang_sub_Dyk}
\bysame, \emph{Subdifferentiable functions and partial data communication in a
  distributed deterministic asynchronous {D}ykstra's algorithm}, 2018.

\bibitem[PXYY16]{AROCK_Peng_Xu_Yan_Yin}
Z.~Peng, Y.~Xu, M.~Yan, and W.~Yin, \emph{{AROCK}: An algorithmic framework for
  asynchronous parallel coordinate updates}, SIAM J. Sci. Comput. \textbf{38}
  (2016), no.~5, A2851--A2879.

\bibitem[SLWY15]{EXTRA_Shi_Ling_Wu_Yin}
W.~Shi, Q.~Ling, G.~Wu, and W.~Yin, \emph{{EXTRA}: An exact first-order
  algorithm for decentralized consensus optimization}, SIAM J. Optim.
  \textbf{25} (2015), no.~2, 944--966.

\bibitem[VHDG11]{Vaidya_Hadji_Domin_2011}
N.H. Vaidya, C.N. Hadjicostis, and A.D. Dominguez-Garcia, \emph{Distributed
  algorithms for consensus and coordination in the presence of packet-dropping
  communication links-part ii: {C}oefficients of ergodicity analysis approach},
  arXiv preprint arXiv:1109.6392, 2011.

\bibitem[WO13]{Wei_Ozdaglar_2013}
E.~Wei and A.~Ozdaglar, \emph{On the {O}(1/k) convergence of asynchronous
  distributed alternating direction method of multipliers}, Proceedings of the
  2013 IEEE Global Conference on Signal and Information Processing (GlobalSIP),
  IEEE Press, Piscataway, NJ, 2013, pp.~551--554.

\end{thebibliography}

\end{document}